\documentclass[12pt]{article}
\pdfoutput=1

\usepackage{amsmath, amsthm, amssymb}
 \usepackage{ifpdf}
\ifpdf
\usepackage[pdftex]{graphicx}
\else
\usepackage[dvips]{graphicx}
\fi
\usepackage{tikz}
 	 \usetikzlibrary{arrows,backgrounds} 
\usepackage[all]{xy}

\IfFileExists{mypreamble.tex}{\input{mypreamble.tex}}{%auto-ignore
\usepackage{amsmath,amssymb,amsfonts,amsthm}
\usepackage{ifpdf}
\usepackage{enumerate}
\usepackage{ulem}
\usepackage{leftidx}

\usepackage{breqn}

\usepackage{tikz}
\usetikzlibrary{calc}
\usepackage{tikz-qtree}

\ifpdf
\usepackage[pdftex,plainpages=false,hypertexnames=false,pdfpagelabels]{hyperref}
\else
\usepackage[dvips,plainpages=false,hypertexnames=false]{hyperref}
\fi

\vfuzz2pt % Don't report over-full v-boxes if over-edge is small
\hfuzz2pt % Don't report over-full h-boxes if over-edge is small
\newcommand{\arxiv}[1]{\href{http://arxiv.org/abs/#1}{\tt arXiv:\nolinkurl{#1}}}
\newcommand{\doi}[1]{\href{http://dx.doi.org/#1}{{\tt DOI:#1}}}

\newcommand{\googlebooks}[1]{(preview at \href{http://books.google.com/books?id=#1}{google books})}

\theoremstyle{plain}
\newtheorem{prop}{Proposition}[section]

\newtheorem{thm}[prop]{Theorem}
\newtheorem{thm*}{Theorem}

\newtheorem{lem}[prop]{Lemma}
\newtheorem{cor}[prop]{Corollary}
\newtheorem*{cor*}{Corollary}
\newtheorem*{clm*}{Claim}
\newtheorem*{conj*}{Conjecture}
\numberwithin{equation}{section}
\theoremstyle{remark}
\newtheorem{ex}[prop]{Example}

\newtheorem{remark}[prop]{Remark}           
\newtheorem*{rem*}{Remark}               %unnumbered remark
\newtheorem*{ex*}{Example}                %unnumbered exercise

\theoremstyle{definition}
\newtheorem{defn}[prop]{Definition}         % numbered definition
   
\newtheorem{nota}[prop]{Notation}   
\newtheorem*{defn*}{Definition}             % unnumbered definition

\theoremstyle{plain}

\newcounter{comment}
\newcommand{\noop}[1]{}

\def\clap#1{\hbox to 0pt{\hss#1\hss}}

\newcommand{\Natural}{\mathbb N}

\def\semicolon{;}
\def\applytolist#1{
    \expandafter\def\csname multi#1\endcsname##1{
        \def\multiack{##1}\ifx\multiack\semicolon
            \def\next{\relax}
        \else
            \csname #1\endcsname{##1}
            \def\next{\csname multi#1\endcsname}
        \fi
        \next}
    \csname multi#1\endcsname}

\def\calc#1{\expandafter\def\csname c#1\endcsname{{\mathcal #1}}}
\applytolist{calc}QWERTYUIOPLKJHGFDSAZXCVBNM;
\def\bbc#1{\expandafter\def\csname bb#1\endcsname{{\mathbb #1}}}
\applytolist{bbc}QWERTYUIOPLKJHGFDSAZXCVBNM;
\def\bfc#1{\expandafter\def\csname bf#1\endcsname{{\mathbf #1}}}
\applytolist{bfc}QWERTYUIOPLKJHGFDSAZXCVBNM;

\renewcommand{\imath}{\mathfrak{i}}
\renewcommand{\jmath}{\mathfrak{j}}

\newcommand{\set}[1]{\left\{#1\right\}}

\makeatletter
\newcommand{\hashdef}[2]{\@namedef{#1}{#2}}
\newcommand{\hashlookup}[1]{\@nameuse{#1}}
\makeatother

\IfFileExists{../../graphs/lookup.tex}%
	{\newcommand{\pathtographs}{../../graphs/}}%
	{\newcommand{\pathtographs}{diagrams/graphs/}}

\hashdef{bwd1v1p1v1x1p0x1duals1v1x2}{3C97CB7FBDEFB761}%
\hashdef{bwd1v1v1p1p1v1x0x0p0x1x0duals1v1v2x1}{63FDB4A04F409318}%
\hashdef{bwd1v1v1p1v1x0p0x1p0x1v0x1x0p1x0x1duals1v1v2x1x3}{185DF41920CC8F1E}%
\hashdef{bwd1v1v1v1p1p1v1x0x0p0x1x0p0x0x1v1x0x0p0x1x0p0x0x1duals1v1v1x2x3v1x3x2}{B5425B05363851E6}%
\hashdef{bwd1v1v1v1p1p1v1x0x0p0x1x0p0x1x0v1x0x0duals1v1v1x2x3v1}{DB2C121504201C9D}%
\hashdef{bwd1v1v1v1p1v1x0p0x1v1x0p0x1duals1v1v1x2v2x1}{25781EDEA61320E6}%
\hashdef{bwd1v1v1v1p1v1x0p1x0duals1v1v1x2}{E26689FBED470E90}%
\hashdef{bwd1v1v1v1v1v1p1v0x1p0x1v0x1v1duals1v1v1v1x2v1}{96B12B0FC207748A}%
\hashdef{bwd1v1v1v1v1v1p1v1x0p0x1v1x0p0x1p0x1v1x0x0v1duals1v1v1v1x2v2x1x3v1}{1D8CDC2F9BD90A57}%
\hashdef{bwd1v1v1v1v1v1v1v1p1v1x0p0x1v1x0p0x1duals1v1v1v1v1x2v2x1}{C53E32050C862F37}%
\hashdef{bwd1v1v1v1v1v1v1v1p1v1x0p1x0duals1v1v1v1v1x2}{CA1952BED6C8EB18}%
\hashdef{gbg1v1p1v1x0p0x1}{14CE5A729ED34B23}%
\hashdef{gbg1v1p1v1x0p0x1v1x0p0x1}{B6129649266AB462}%
\hashdef{gbg1v1p1v1x0p0x1v1x0p1x0p0x1}{24940C555ECF7093}%
\hashdef{gbg1v1p1v1x0p1x0}{52EBF913415F8806}%
\hashdef{gbg1v1v1p1p1v1x0x0p0x1x0}{688C471F1900888A}%
\hashdef{gbg1v1v1v1p1p1v1x0x0p0x1x0p0x0x1v1x0x0p0x1x0p0x0x1}{BAF17D013FA99673}%
\hashdef{gbg1v1v1v1p1p1v1x0x0v1}{D9E55F89702ABF01}%
\hashdef{gbg1v1v1v1p1v1x0p0x1}{34956B109EE640A6}%
\hashdef{gbg1v1v1v1p1v1x0p1x0}{C5F8B0506196E46A}%
\hashdef{gbg1v1v1v1v1p1p1v1x0x0p0x1x0p0x0x1v0x1x0p0x0x1v1x0p0x1}{81D9CE06CA7DF0A9}%

\newcommand{\bigraph}[1]{{\hspace{-3pt}\begin{array}{c}%
  \raisebox{-2.5pt}{\includegraphics[height=6mm]{\pathtographs \hashlookup{#1}}}% 
\end{array}\hspace{-3pt}}}

\newcommand{\Hom}[3]{\operatorname{Hom_{#1}}\!\!\left(#2 \to #3\right)}

\tikzstyle{shaded}=[fill=red!10!blue!20!gray!30!white]
\tikzstyle{unshaded}=[fill=white]
\tikzstyle{empty box}=[circle, draw, thick, fill=white, opaque, inner sep=2mm]
\tikzstyle{annular}=[scale=.7, inner sep=1mm, baseline]
\tikzstyle{rectangular}=[scale=.75, inner sep=1mm, baseline=-.1cm]

\textwidth   5.5in%
\textheight  9.0in%
\oddsidemargin 12pt%
\evensidemargin 12pt

\topmargin -.6in%
\headsep .5in
}

\usepackage{xifthen}  % http://ctan.org/pkg/xifthen
\usepackage{breqn}

\tikzstyle{rectangular}=[scale=.75, inner sep=1mm, baseline=-.1cm]

\usepackage{xcolor}
\definecolor{dark-red}{rgb}{0.7,0.25,0.25}
\definecolor{dark-blue}{rgb}{0.15,0.15,0.55}
\definecolor{medium-blue}{rgb}{0,0,0.65}
\hypersetup{
   colorlinks, linkcolor={purple},
   citecolor={medium-blue}, urlcolor={medium-blue}
}

\DeclareMathOperator{\spann}{span}

\DeclareMathOperator{\trains}{trains}
\DeclareMathOperator{\Tr}{Tr}
\newcommand{\shaded}{
\begin{tikzpicture}[baseline = -.13cm]
	\filldraw[shaded] (0,0) circle (.1cm);
\end{tikzpicture}
}
\newcommand{\unshaded}{
\begin{tikzpicture}[baseline = -.13cm]
	\filldraw[unshaded] (0,0) circle (.1cm);
\end{tikzpicture}
}
\renewcommand{\set}[2]{\left\{#1\middle|#2\right\}}

\newcommand{\Afinite}{A_{\mathit{finite}}}

\newcommand{\quadratic}[3]{
\begin{tikzpicture}[baseline = -.3cm]
	\draw (0,-.8)--(0,0)--(1.6,0)--(1.6,-.8);
	\node at (.8,.15) {{\scriptsize{$#1-1$}}};
	\node at (.4,-.5) {{\scriptsize{$#1+1$}}};
	\node at (2,-.5) {{\scriptsize{$#1+1$}}};
	\draw[thick, unshaded] (0,0) circle (.4);
	\node at (0,0) {$#2$};
	\draw[thick, unshaded] (1.6,0) circle (.4);
	\node at (1.6,0) {$#3$};
	\node at (1.6,.55) {$\star$};
	\node at (0,.55) {$\star$};
\end{tikzpicture}
}
\newcommand{\jellyfish}[1]{
\begin{tikzpicture}[baseline = -.2cm]
	\filldraw[unshaded] (-.7,-.8)--(-.7,0) arc (180:0:.7cm)--(.7,-.8);
	\draw (0,0)--(0,-.8);
	\node at (.3,-.6) {{\scriptsize{$2n$}}};
	\draw[thick, unshaded] (0,0) circle (.4);
	\node at (0,0) {$#1$};
	\node at (-.25,-.5) {$\star$};
\end{tikzpicture}
}
\newcommand{\jellyfishSquared}[1]{
\begin{tikzpicture}[baseline = -.3cm]
	\filldraw[shaded] (-.9,-.8)--(-.9,0) arc (180:0:.9cm)--(.9,-.8);
	\filldraw[unshaded] (-.7,-.8)--(-.7,0) arc (180:0:.7cm)--(.7,-.8);
	\draw (0,0)--(0,-.8);
	\node at (.2,-.6) {{\scriptsize{$2n$}}};
	\draw[thick, unshaded] (0,0) circle (.4);
	\node at (0,0) {$#1$};
	\node at (-.25,-.5) {$\star$};
\end{tikzpicture}
}
\newcommand{\MR}[1]{}

\begin{document}

\title{Principal graph stability and the jellyfish algorithm}
\author{Stephen Bigelow and David Penneys}

\maketitle
\begin{abstract}
We show that if the principal graph of a subfactor planar algebra of modulus $\delta>2$ is stable for two depths, then it must end in $\Afinite$ tails. This result is analogous to Popa's theorem on principal graph stability. We use these theorems to show that an $n-1$ supertransitive subfactor planar algebra has jellyfish generators at depth $n$ if and only if its principal graph is a spoke graph. 
%This is the published version of \todo{arXiv version}.
\end{abstract}

%MSC2010: 46L37, 18D05, 57M20

%%%%%%%%%%%%%%%%%%%%%%%%%%%%%%%%%%%%%%%%%%%%%%%%%%%%%%%%%%%%%%%%%%%%%%%%%%
%%%%%%%%%%%%%%%%%%%%%%%%%%%%%%%%%%%%%%%%%%%%%%%%%%%%%%%%%%%%%%%%%%%%%%%%%%
%%%%%%%%%%%%%%%%%%%%%%%%%%%%%%%%%%%%%%%%%%%%%%%%%%%%%%%%%%%%%%%%%%%%%%%%%%
\section{Introduction}

Every subfactor planar algebra embeds in the graph planar algebra of its principal graph \cite{MR2812459,tvc}. 
Thus one can construct a subfactor planar algebra by finding candidate generators in the appropriate graph planar algebra, and then showing the planar algebra they generate is a subfactor planar algebra with the correct principal graph. 
Since a graph planar algebra satisfies all the unitarity conditions of a subfactor planar algebra, one must only show the planar subalgebra $P_\bullet$ is \emph{evaluable}, i.e., $\dim(P_{0,\pm})=1$, to get some subfactor planar algebra.  (Additional arguments are needed to verify the principal graph of $P_\bullet$ is the graph with which we started.)

The \emph{jellyfish algorithm} of \cite{0909.4099} is an evaluation algorithm with two main ingredients:
\begin{enumerate}
\item[(1)] Elements in a set of generators $\cS_\pm\subseteq P_{n,\pm}$ satisfy \emph{jellyfish relations}, i.e., diagrams like 
$$
j(\check{S_1})=
\jellyfish{\check{S_1}}
\,,\, 
j^2(S_2)= 
\jellyfishSquared{S_2}\,,
$$
where $\check{S_1}\in \cS_-,S_2\in \cS_+$, can be written as linear combinations of \emph{trains}, which are diagrams where any region meeting the distinguished interval of a generator meets the distinguished interval of the external disk, e.g.,
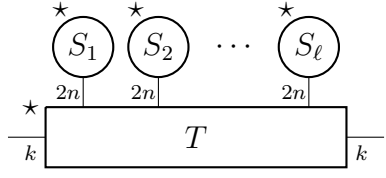
$$
\begin{tikzpicture}[baseline = -.3cm]
	\draw (-2,-1.2)--(3,-1.2);
	\draw (-1,0)--(-1,-1);
	\draw (0,0)--(0,-1);
	\draw (2,0)--(2,-1);
	\filldraw[thick, unshaded] (-1.5,-.8)--(-1.5,-1.6)--(2.5,-1.6)--(2.5,-.8)--(-1.5,-.8);
	\draw[thick, unshaded] (-1,0) circle (.4);
	\draw[thick, unshaded] (0,0) circle (.4);	
	\draw[thick, unshaded] (2,0) circle (.4);
	\node at (-1.7,-.8) {$\star$};
	\node at (-1.3,.5) {$\star$};	
	\node at (-.3,.5) {$\star$};
	\node at (1.7,.5) {$\star$};
	\node at (1,0) {$\cdots$};
	\node at (.5,-1.2) {$T$};
	\node at (-1,0) {$S_1$};
	\node at (0,0) {$S_2$};
	\node at (2,0) {$S_\ell$};
	\node at (-1.7,-1.4) {{\scriptsize{$k$}}};
	\node at (2.7,-1.4) {{\scriptsize{$k$}}};
	\node at (-1.2,-.6) {{\scriptsize{$2n$}}};
	\node at (-.2,-.6) {{\scriptsize{$2n$}}};
	\node at (1.8,-.6) {{\scriptsize{$2n$}}};
\end{tikzpicture}
$$
where $S_1,\dots, S_\ell\in \cS_\pm$ and $T$ is a single Temperley-Lieb diagram (we suppress the external disk, and the external star goes in the upper left corner).
\item[(2)] The generators in $\cS_\pm$, together with the Jones-Wenzl projection $f^{(n)}$, form an algebra under the usual multiplication 
$$
\begin{tikzpicture}[baseline = .8cm]
	\draw (0,2.2)--(0,-.4);	
	\filldraw[unshaded,thick] (0,1.5) circle (.4cm);
	\node at (0,1.5) {$S_j$};
	\node at (-.55,1.5) {$\star$};
	\filldraw[unshaded, thick] (0,.3) circle (.4cm);
	\node at (0,.3) {$S_i$};
	\node at (-.55,.3) {$\star$};
	\node at (-.2,.9) {{\scriptsize{$n$}}};
	\node at (-.2,-.2) {{\scriptsize{$n$}}};
	\node at (-.2,2) {{\scriptsize{$n$}}};
\end{tikzpicture}
=
\sum_{R}
\lambda_{i,j}^k
\begin{tikzpicture}[baseline = -.1cm]
	\draw (0,.8)--(0,-.8);
	\filldraw[unshaded,thick] (0,0) circle (.4cm);
	\node at (0,0) {$S_k$};
	\node at (-.55,0) {$\star$};
	\node at (-.2,.6) {{\scriptsize{$n$}}};
	\node at (-.2,-.6) {{\scriptsize{$n$}}};
\end{tikzpicture}\,.
$$
\end{enumerate}
Given these two ingredients, one can evaluate any closed diagram in two steps.
\begin{enumerate}
\item[(1)] Pull all generators $S$ to the outside of the diagram using the jellyfish relations, possibly getting diagrams with more $S$'s, and
\item[(2)] Iteratively reduce the number of generators using the algebra property and an inner-most disk argument.
\end{enumerate}

The jellyfish algorithm was first used in \cite{0909.4099} to construct the extended Haagerup subfactor planar algebra with principal graphs
$$
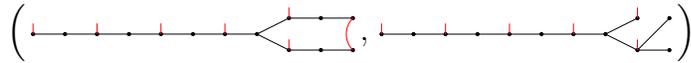
\left(\bigraph{bwd1v1v1v1v1v1v1v1p1v1x0p0x1v1x0p0x1duals1v1v1v1v1x2v2x1}, \bigraph{bwd1v1v1v1v1v1v1v1p1v1x0p1x0duals1v1v1v1v1x2}\right)
$$
(the red markings at the even depths give the dual data), which completed the classification of non-$A_\infty$ subfactors in the index range $(4,3+\sqrt{3})$. They found \emph{$2$-strand jellyfish relations}
$$
j(\check{S})\in\spann(\trains_{5,+}(\{S\})) \text{ and }j^2(S)\in \spann(\trains_{6,+}(\{S\})) 
$$
to evaluate all diagrams that are unshaded on the outside (see Definition \ref{defn:TwoStrandJellyfishRelations} for the relevant notation).

The algorithm was used again in Han's thesis \cite{1102.2052} to give a planar algebra construction of the Izumi-Xu 2221 subfactor planar algebra with principal graphs
$$
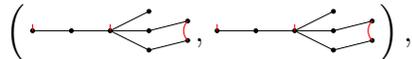
\left(\bigraph{bwd1v1v1p1p1v1x0x0p0x1x0duals1v1v2x1}, \bigraph{bwd1v1v1p1p1v1x0x0p0x1x0duals1v1v2x1}\right),
$$
but with simpler \emph{$1$-strand jellyfish relations}:
\begin{align*}
j(\check{S_1}),j(\check{S_2})&\in \spann(\trains_{4,+}(\{S_1,S_2\})) \text{ and } \\
j(S_1),j(S_2)&\in \spann(\trains_{4,-}(\{\check{S_1},\check{S_2}\})).
\end{align*}
(Note that these relations immediately imply relations for $j^2(S_i)$, $i=1,2$).

In recent work \cite{MPspokes}, Morrison and Penneys use the jellyfish algorithm to automate the construction of certain subfactor planar algebras whose principal and dual principal graphs are \emph{spoke graphs}, which are trees with at most one vertex of degree greater than 2 (possibly with some multiple edges near the central vertex. See Definition \ref{defn:spoke}). 
They constructed  a new 4442 spoke subfactor along with a number of known spoke subfactors, including the Izumi-Xu 2221 (automating Han's thesis), the Goodman-de la Harpe-Jones 3311, and the Izumi 3333.% ($3^{\Integer/2\times\Integer/2}$, not $3^{\Integer/4}$ as in Remark \ref{rem:3333Zmod4}). 
Again, simpler $1$-strand jellyfish relations were found.

Bigelow-Morrison-Peters-Snyder \cite{0909.4099} noticed that $1$-strand jellyfish generators did not exist for the (extended) Haagerup subfactor planar algebra. Morrison and Penneys also noticed their non-existence for all known examples of subfactor planar algebras with annular multiplicities $*10$, i.e., for which the principal graphs $(\Gamma_+,\Gamma_-)$ are a translated extension of 
$$
\left(\bigraph{gbg1v1p1v1x0p0x1},\bigraph{gbg1v1p1v1x0p1x0}\right)
\text{ or }
\left(\bigraph{gbg1v1p1v1x0p1x0},\bigraph{gbg1v1p1v1x0p0x1}\right)
$$
(\emph{translating} a principal graph means attaching an $A_k$ graph to the left, and \emph{extending} means adding additional edges and vertices to the right). 
For more details on annular multiplicities $*10$, see \cite{MR1317352,math/1007.1158,1007.2240}. 

In this paper, we classify exactly which subfactor planar algebras can be constructed using the jellyfish algorithm.

\begin{thm}\label{thm:JellyfishAndSpokes}
A $n-1$ supertransitive subfactor planar algebra can be constructed using jellyfish generators at depth $n$ if and only if its principal graph is a spoke graph. 
We can find 1-strand jellyfish generators if and only if both the principal graph and dual principal graph are spoke graphs.
See Theorems \ref{thm:OneStrand} and \ref{thm:TwoStrand} for more details.
\end{thm}

To prove this result, we use techniques from Section 4 of Popa's paper \cite{MR1334479}. Popa calls a (dual) principal graph $\Gamma$ \emph{stable at depth $n$} if $\Gamma$ does not merge or split between depths $n$ and $n+1$, and all edges between depths $n$ and $n+1$ are simple. He proves a remarkable result, which we call \emph{Popa's Principal Graph Stability Theorem}. 
For context, let $P_\bullet$ be a subfactor planar algebra of modulus $\delta$ with principal graphs $(\Gamma_+,\Gamma_-)$, and let $\Gamma_\pm(k)$ denote the truncation of $\Gamma_\pm$ to depth $k$.

\begin{thm}[Popa's Principal Graph Stability Theorem 4.5 of \cite{MR1334479}]\label{thm:PrincipalGraphStability}
If $(\Gamma_+,\Gamma_-)$ is stable at depth $n$, the truncation $\Gamma_\pm(n+1)\neq A_{n+2}$, and $\delta>2$, then $(\Gamma_+,\Gamma_-)$ is stable at depth $k$ for all $k\geq n$, and $\Gamma_+,\Gamma_-$ are finite.
\end{thm}

In examining this theorem, we found that trains first appeared in \cite{MR1334479} in the language of $\lambda$-lattices! Using Popa's techniques along with trains and ideas stemming from the jellyfish algorithm, we prove an analogous result only looking at the principal graph, which is a strengthening of (a) of Lemma 4.7 of \cite{MR1334479}. 

\begin{thm}\label{thm:EvenStability}
If $\Gamma_+$ is stable at depths $n$ and $n+1$, the truncation $\Gamma_+(n+1)\neq A_{n+2}$, and $\delta>2$, then $(\Gamma_+,\Gamma_-)$ is stable at depth $k$ for all $k\geq n+1$, and $\Gamma_+,\Gamma_-$ are finite.
\end{thm}

Planar algebras are essential to our approach. We use the 1-click rotation (also known as the Fourier Transform), which is natural from a planar algebra viewpoint, in the important Lemma \ref{lem:FourierTransform}.

One of the biggest hurdles in the classification of subfactors to index 5 \cite{1007.1730,1007.2240,1109.3190,1010.3797} were weeds with initial quadruple points. (A \emph{weed} represents an infinite family of potential principal graphs obtained from a fixed subgraph by translating and extending. See \cite{1007.1730} for more details.) Arguments to rule out $\cQ,\cQ'$ in \cite{1109.3190} were case specific; they knew no general techniques for quadruple points to go beyond index 5. The theorems in this paper and \cite{MR1334479} not only simplify eliminating $\cQ,\cQ'$ in \cite{1109.3190} (and $\cB$ in \cite{1007.2240}), but also eliminate all remaining weeds with initial quadruple points up to index $3+\sqrt{5}$, providing more evidence for \cite[Conjecture 2.2]{1205.2742} of Morrison-Peters:
\begin{conj*}
Any subfactor with index in the range $(5,3+\sqrt{5})$ has principal graphs $(A_\infty,A_\infty)$, 
$
\left(
\bigraph{bwd1v1p1v1x1p0x1duals1v1x2},
\bigraph{bwd1v1p1v1x1p0x1duals1v1x2}
\right)
$,
or
$
\left(
\bigraph{bwd1v1v1p1v1x0p0x1p0x1v0x1x0p1x0x1duals1v1v2x1x3},
\bigraph{bwd1v1v1p1v1x0p0x1p0x1v0x1x0p1x0x1duals1v1v2x1x3}
\right)
$.
\end{conj*}
Using our results, Morrison and Penneys have shown that to prove the conjecture of Morrison-Peters, one needs to eliminate roughly 10 weeds with initial triple points. These new weeds are similar to weeds eliminated in \cite{1007.1730,1007.2240}, but they are more complex.

Numerous other applications of our results are given in Section \ref{sec:Applications}. We anticipate that our results will prove strong new obstructions to possible principal graphs. 

%%%%%%%%%%%%%%%%%%%%%%%%%%%%%%%%%%%%%%%%%%
\subsection{Outline}

Section \ref{sec:Background} contains the background for this paper.
Subsection \ref{sec:PrincipalGraphs} briefly recalls how to get a rigid, unitary, spherical 2-category $\cG(P_\bullet)$ from a subfactor planar algebra $P_\bullet$ and how to define the principal graphs $(\Gamma_+,\Gamma_-)$ from $\cG(P_\bullet)$. 
Subsection \ref{sec:Popa} gives Popa's definition of stability for planar algebras and principal graphs and shows they are compatible.

In Section \ref{sec:ExtendingPopa}, we go through the proof of Popa's Theorem \ref{thm:PrincipalGraphStability} using planar algebras and trains to prove Theorem \ref{thm:EvenStability}. In Subsection \ref{sec:Trains}, we define trains, and we prove the important Lemma  \ref{lem:FourierTransform}. 
In Subsection \ref{sec:TrainsAndStability}, we show that stability is equivalent to trains spanning. 
Subsection \ref{sec:TrainsAndJellyfish} connects trains and jellyfish, and
Subsection \ref{sec:Proof} contains the proof of Theorem \ref{thm:EvenStability}.

In Section \ref{sec:Applications}, we give a number of applications of our results.  
Subsection \ref{sec:Jellyfish} explains the connection between the jellyfish algorithm and spoke principal graphs, proving Theorem \ref{thm:JellyfishAndSpokes}. 
Afterward, we give a few quick corollaries and a remark which uses the classification of subfactors to index 5 \cite{1007.1730,1007.2240,1109.3190,1010.3797} to classify all simply laced, acyclic principal graphs of subfactors with at most 2 triple points. 
Subsection \ref{sec:QT} gives a simple proof of Jones' quadratic tangles obstruction for annular multiplicities $*10$ subfactor planar algebras. 

%%%%%%%%%%%%%%%%%%%%%%%%%%%%%%%%%%%%%%%%%%%%%%%%%%%%%%%%%%%%%%%%%%%%%%%%%%
%%%%%%%%%%%%%%%%%%%%%%%%%%%%%%%%%%%%%%%%%%%%%%%%%%%%%%%%%%%%%%%%%%%%%%%%%%
%%%%%%%%%%%%%%%%%%%%%%%%%%%%%%%%%%%%%%%%%%%%%%%%%%%%%%%%%%%%%%%%%%%%%%%%%%
\section{Background}\label{sec:Background}

We refer the reader to \cite{math/1007.1158,0909.4099,JonesPANotes} for the definition of a (subfactor) planar algebra. 

\begin{nota}
When we draw planar diagrams, we often suppress the external boundary disk. 
In this case, the external boundary is assumed to be a large rectangle whose distinguished interval contains the upper left corner. 
We draw one string with a number next to it instead of drawing that number of parallel strings. 
We shade the diagrams as much as possible, but if the parity is unknown, we  often cannot know how to shade them. 
Finally, projections are usually drawn as rectangles with the same number of strands emanating from the top and bottom, while other elements may be drawn as circles.
\end{nota}

%%%%%%%%%%%%%%%%%%%%%%%%%%%%%%%%%%%%%%%%%%%%%%%%%%%%%%%%%%
\subsection{$2$-categories and fusion graphs}\label{sec:PrincipalGraphs}

We recall how to get a rigid, unitary, spherical 2-category $\cG(P_\bullet)$ from a subfactor planar algebra $P_\bullet$ and how to define the principal graphs $(\Gamma_+,\Gamma_-)$ from $\cG(P_\bullet)$ (see also Section 4.1 of \cite{MR2559686}). 

\begin{defn}
The \emph{paragroup} $\cG(P_\bullet)$ of $P_\bullet$, a rigid, unitary, spherical $2$-category, is defined as follows.

The \emph{objects} of $\cG(P_\bullet)$ are the symbols $\unshaded$ and $\shaded$.

The \emph{$1$-morphisms} of $\cG(P_\bullet)$ are the projections of $P_\bullet$.
\begin{align*}
\Hom{\cG(P_\bullet)}{\unshaded}{\unshaded}&=\set{p\in P_{i,+}}{p \text{ is a projection and } i \text{ is even}},\\
\Hom{\cG(P_\bullet)}{\unshaded}{\shaded}&=\set{p\in P_{i,+}}{p \text{ is a projection and } i \text{ is odd}},\\
\Hom{\cG(P_\bullet)}{\shaded}{\unshaded}&=\set{p\in P_{i,-}}{p \text{ is a projection and } i \text{ is odd}},\text{ and}\\
\Hom{\cG(P_\bullet)}{\shaded}{\shaded}&=\set{p\in P_{i,-}}{p \text{ is a projection and } i \text{ is even}}.
\end{align*}
The identity $1$-morphisms are the empty diagrams. Composition of $1$-morphisms, denoted $\otimes$, is given by horizontal concatenation; e.g., if $p\in\Hom{\cG(P_\bullet)}{\unshaded}{\shaded}$ and $q\in \Hom{\cG(P_\bullet)}{\shaded}{\unshaded}$, then 
$$
\begin{tikzpicture}[baseline = -.1cm]
	\fill[shaded] (0,.8)--(0,-.8)--(.8,-.8)--(.8,.8);
	\draw (0,.8)--(0,-.8);
	\filldraw[unshaded,thick] (-.6,.4)--(.6,.4)--(.6,-.4)--(-.6,-.4)--(-.6,.4);
	\node at (0,0) {$p\otimes q$};
\end{tikzpicture}
=
\begin{tikzpicture}[baseline = -.1cm]
	\fill[shaded] (0,.8)--(0,-.8)--(.7,-.8)--(.7,.8);
	\draw (0,.8)--(0,-.8);
	\filldraw[unshaded,thick] (-.4,.4)--(.4,.4)--(.4,-.4)--(-.4,-.4)--(-.4,.4);
	\node at (0,0) {$p$};
\end{tikzpicture}
\otimes 
\begin{tikzpicture}[baseline = -.1cm]
	\fill[shaded] (-.7,.8)--(-.7,-.8)--(0,-.8)--(0,.8);
	\draw (0,.8)--(0,-.8);
	\filldraw[unshaded,thick] (-.4,.4)--(.4,.4)--(.4,-.4)--(-.4,-.4)--(-.4,.4);
	\node at (0,0) {$q$};
\end{tikzpicture}
=
\begin{tikzpicture}[baseline = -.1cm]
	\fill[shaded] (0,.8)--(0,-.8)--(1.2,-.8)--(1.2,.8);
	\draw (0,.8)--(0,-.8);
	\draw (1.2,.8)--(1.2,-.8);
	\filldraw[unshaded,thick] (-.4,.4)--(.4,.4)--(.4,-.4)--(-.4,-.4)--(-.4,.4);
	\filldraw[unshaded,thick] (.8,.4)--(1.6,.4)--(1.6,-.4)--(.8,-.4)--(.8,.4);
	\node at (0,0) {$p$};
	\node at (1.2,0) {$q$};
\end{tikzpicture}\,.
$$
A $1$-morphism $p$ is called \emph{simple} if $\dim(\Hom{\cG(P_\bullet)}{p}{p})=1$.

The \emph{$2$-morphisms} of $\cG$ are as follows.
If $p_1\in P_{i,\pm}$ and $p_2\in P_{j,\pm}$
then $\Hom{\cG(P_\bullet)}{p_1}{p_2}$ is $p_2 P_{j\to i}{p_1}$,
where $P_{j\to i}$ is $P_{i+j}$
with $j$ strings on the bottom and $i$ strings on the top.
Note that $\Hom{\cG(P_\bullet)}{p_1}{p_2}=(0)$
if $i$ and $j$ do not have the same parity.
$$
\begin{tikzpicture}[baseline = 0cm]
	\clip (-.8,-2)--(-.8,2)--(.8,2)--(.8,-2);
	\draw (0,2)--(0,-2);	
	\node at (-.2,1.8) {{\scriptsize{$j$}}};
	\node at (-.2,.6) {{\scriptsize{$j$}}};
	\node at (-.2,-.6) {{\scriptsize{$i$}}};
	\node at (-.2,-1.8) {{\scriptsize{$i$}}};
	\filldraw[unshaded,thick] (-.4,.8)--(.4,.8)--(.4,1.6)--(-.4,1.6)--(-.4,.8);
	\node at (0,1.2) {$p_2$};
	\filldraw[unshaded,thick] (0,0) circle (.4cm);
	\node at (0,0) {$?$};
	\node at (-.55,0) {$\star$};
	\filldraw[unshaded,thick] (-.4,-.8)--(.4,-.8)--(.4,-1.6)--(-.4,-1.6)--(-.4,-.8);
	\node at (0,-1.2) {$p_1$};
\end{tikzpicture}
\in \Hom{\cG(P_\bullet)}{p_1}{p_2}.
$$
The two types of composition of $2$-morphisms are given by vertical and horizontal concatenation of diagrams. If we have $x\in\Hom{\cG(P_\bullet)}{p_1}{p_2}$ and $y\in \Hom{\cG(P_\bullet)}{p_2}{p_3}$, then the vertical multiplication $xy$ is given by
$$
\begin{tikzpicture}[baseline = -.1cm]
	\draw (0,.7)--(0,-.7);
	\filldraw[unshaded,thick] (-.4,.4)--(.4,.4)--(.4,-.4)--(-.4,-.4)--(-.4,.4);
	\node at (0,0) {$xy$};
\end{tikzpicture}
=
\begin{tikzpicture}[baseline = .8cm]
	\draw (0,2.4)--(0,-.6);	
	\filldraw[unshaded,thick] (-.4,1.2)--(.4,1.2)--(.4,2)--(-.4,2)--(-.4,1.2);
	\node at (0,1.6) {$y$};
	\filldraw[unshaded,thick] (-.4,-.2)--(.4,.-.2)--(.4,.6)--(-.4,.6)--(-.4,-.2);
	\node at (0,.2) {$x$};
\end{tikzpicture}\,.
$$
If $x\in\Hom{\cG(P_\bullet)}{p_1}{p_2}$ and $y\in \Hom{\cG(P_\bullet)}{p_3}{p_4}$ and $p_1,p_2$ are composable with $p_3,p_4$ respectively, then the horizontal multiplication $x\otimes y$ is given by
$$
\begin{tikzpicture}[baseline = -.1cm]
	\draw (0,.7)--(0,-.7);
	\filldraw[unshaded,thick] (-.6,.4)--(.6,.4)--(.6,-.4)--(-.6,-.4)--(-.6,.4);
	\node at (0,0) {$x\otimes y$};
\end{tikzpicture}
=
\begin{tikzpicture}[baseline = -.1cm]
	\draw (0,.7)--(0,-.7);
	\filldraw[unshaded,thick] (-.4,.4)--(.4,.4)--(.4,-.4)--(-.4,-.4)--(-.4,.4);
	\node at (0,0) {$x$};
\end{tikzpicture}
\otimes
\begin{tikzpicture}[baseline = -.1cm]
	\draw (0,.7)--(0,-.7);
	\filldraw[unshaded,thick] (-.4,.4)--(.4,.4)--(.4,-.4)--(-.4,-.4)--(-.4,.4);
	\node at (0,0) {$y$};
\end{tikzpicture}
=
\begin{tikzpicture}[baseline = -.1cm]
	\draw (0,.7)--(0,-.7);
	\draw (1.2,.7)--(1.2,-.7);
	\filldraw[unshaded,thick] (-.4,.4)--(.4,.4)--(.4,-.4)--(-.4,-.4)--(-.4,.4);
	\filldraw[unshaded,thick] (.8,.4)--(1.6,.4)--(1.6,-.4)--(.8,-.4)--(.8,.4);
	\node at (0,0) {$x$};
	\node at (1.2,0) {$y$};
\end{tikzpicture}\,.
$$

The \emph{adjoint} operation in $\cG(P_\bullet)$ is
the identity on objects and $1$-morphisms.
The adjoint of a $2$-morphism
is the same as the adjoint operation in the planar algebra $P_\bullet$.
If $x\in \Hom{\cG(P_\bullet)}{p_1}{p_2}$,
where $p_1\in P_{i,\pm}$ and $p_2\in P_{j,\pm}$,
then we can consider $x$ as an element of $P_{i+j,\pm}$,
take the adjoint,
and consider the result $x^*$ as an element of
$\Hom{\cG(P_\bullet)}{p_2}{p_1}$.

The \emph{duality} operation on $\cG(P_\bullet)$ is the identity on all objects.
On $1$-morphisms and $2$-morphisms, duality is rotation by $\pi$.
$$
\begin{tikzpicture}[baseline = -.1cm]
	\clip (-.9,.8)--(-.9,-.9)--(.9,-.9)--(.9,.9);
	\fill[shaded] (-.9,.8)--(-.9,-.9)--(.8,-.9)--(.8,.8);
	\filldraw[unshaded,thick] (-.4,.4)--(.4,.4)--(.4,-.4)--(-.4,-.4)--(-.4,.4);
	\filldraw[unshaded] (0,.4) arc (180:0:.35cm)--(.7,-1)--(1,-1)--(1,1)--(-.7,1)--(-.7,-.4) arc (180:360:.35cm);
	\filldraw[unshaded,thick] (-.4,.4)--(.4,.4)--(.4,-.4)--(-.4,-.4)--(-.4,.4);
	\node at (0,0) {$p$};
\end{tikzpicture}
=
\begin{tikzpicture}[baseline = -.1cm]
	\fill[shaded] (0,.8)--(0,-.8)--(-.9,-.8)--(-.9,.8);
	\draw (0,.8)--(0,-.8);
	\filldraw[unshaded,thick] (-.4,.4)--(.4,.4)--(.4,-.4)--(-.4,-.4)--(-.4,.4);
	\node at (0,0) {$\overline{p}$};
\end{tikzpicture}\,.
$$
\end{defn}

\begin{defn}
The \emph{principal graph} $\Gamma_+$ of $P_\bullet$ is defined as follows.
The even vertices of $\Gamma_+$ are the isomorphism classes of
simple $1$-morphisms in $\Hom{}{\unshaded}{\unshaded}$.
The odd vertices of $\Gamma_+$ are the isomorphism classes of
simple $1$-morphisms in $\Hom{}{\unshaded}{\shaded}$.
The number of edges between 
vertices corresponding to simple projections
$p\in\Hom{}{\unshaded}{\unshaded}$ and $q\in\Hom{}{\unshaded}{\shaded}$,
is
$$
\dim\left(
\Hom{\cG(P_\bullet)}
{
\begin{tikzpicture}[baseline = -.1cm]
	\fill[shaded] (.8,.8)--(.8,-.8)--(.6,-.8)--(.6,.8);
	\draw (0,.8)--(0,-.8);
	\draw (.6,.8)--(.6,-.8);
	\node at (.2,.6) {{\scriptsize{$n$}}};
	\node at (.2,-.6) {{\scriptsize{$n$}}};
	\filldraw[unshaded,thick] (-.4,.4)--(.4,.4)--(.4,-.4)--(-.4,-.4)--(-.4,.4);
	\node at (0,0) {$p$};
\end{tikzpicture}
}{
\begin{tikzpicture}[baseline = -.1cm]
	\fill[shaded] (0,.8)--(0,-.8)--(.8,-.8)--(.8,.8);
	\draw (0,.8)--(0,-.8);
	\node at (.4,.6) {{\scriptsize{$n+1$}}};
	\node at (.4,-.6) {{\scriptsize{$n+1$}}};
	\filldraw[unshaded,thick] (-.4,.4)--(.4,.4)--(.4,-.4)--(-.4,-.4)--(-.4,.4);
	\node at (0,0) {$q$};
\end{tikzpicture}
}\right).
$$ 

The \emph{basepoint} $\star$ of $\Gamma_+$
is the vertex corresponding to the unshaded empty diagram.
The \emph{depth} of a vertex of $\Gamma_+$ is its distance from $\star$.
This is equal to the minimum $n$ such that the vertex is
the equivalence class of a projection $p \in P_{n,+}$.

The \emph{dual principal graph} $\Gamma_-$
is defined in exactly the same way as $\Gamma_+$,
but reversing the roles of
$\shaded$ and $\unshaded$.
The \emph{basepoint} $\star$ of $\Gamma_-$
is the vertex corresponding to the shaded empty diagram.

Our graphs are always drawn with the basepoint $\star$ at the left.
\end{defn}

\begin{remark}
The ``plus or minus" symbol $\pm$ is meant to be read respectively throughout an entire statement.  
\end{remark}

\begin{remark}\label{rem:SimplyLaced}
If $\Gamma_\pm$ is simply laced, and $p\in P_{n,\pm}$ is a minimal projection such that the vertex $[p]$ has depth $n$, then we identify
 $[p]$ with $p$.
 \end{remark}

\begin{defn}\label{defn:BratteliDiagram}
Alternatively, from an operator algebras viewpoint, we can define the (dual) principal graph as the principal part of the Bratteli diagram of the tower of finite dimensional von Neumann algebras $P_\pm=(P_{n,\pm})$, where $P_{n,\pm}$ includes into $P_{n+1,\pm}$ unitally via the right inclusion
$$
\begin{tikzpicture}[baseline = -.1cm]
	\draw (0,.8)--(0,-.8);
	\node at (-.2,.6) {{\scriptsize{$n$}}};
	\node at (-.2,-.6) {{\scriptsize{$n$}}};
	\filldraw[unshaded,thick] (-.4,.4)--(.4,.4)--(.4,-.4)--(-.4,-.4)--(-.4,.4);
	\draw (.6,.8)--(.6,-.8);
\end{tikzpicture}\,.
$$
If $z_{n+1,\pm}$ is the central support of the Jones projection
$$
e_{n,\pm}=
\delta^{-1}\,\,
\begin{tikzpicture}[baseline = -.1cm]
	\draw (0,.6)--(0,-.6);
	\node at (.2,0) {{\scriptsize{$n$}}};
	\draw (.2,.6) arc(180:360:.2cm);
	\draw (.2,-.6) arc(180:0:.2cm);
\end{tikzpicture}\in P_{n+1,\pm},
$$
then for each $n\in\Natural$, 
$$
z_{n+1,\pm}P_{n-1,\pm}\subset z_{n+1,\pm}P_{n,\pm}\subset z_{n+1,\pm}P_{n+1,\pm}
$$ 
is the Jones basic construction of finite dimensional von Neumann algebras \cite{MR696688,MR999799}. Hence the Bratteli diagram of $P_\pm$ between depths $n$ and $n+1$ consists of the reflection of the Bratteli diagram between depths $n-1$ and depth $n$, which is referred to as the ``old part," and a ``new part," which can be identified with the Bratteli diagram of the inclusion 
$$
(1-z_{n+1,\pm})P_{n,\pm}\subset (1-z_{n+1,\pm})P_{n+1,\pm}.
$$ 
The principal graph is formed from only the ``new parts." See  \cite{MR999799} for more details.
\end{defn}
 
%%%%%%%%%%%%%%%%%%%%%%%%%%%%%%%%%%%%%%%%%%%
\subsection{Popa's Stability Criterion}\label{sec:Popa}

In \cite[Section 4]{MR1334479}, Popa gives a stability criterion
for $\lambda$-lattices that has very strong consequences.
We define the criterion,
summarize the proof,
and list some consequences.

Let $P_\bullet$ be a subfactor planar algebra,
let $P_\pm=(P_{n,\pm})$ be the respective towers of algebras,
and let $(\Gamma_+,\Gamma_-)$ be the principal and dual principal graphs.
Let $TL_\bullet\subset P_\bullet$ be the Temperley-Lieb planar subalgebra.

\begin{defn}
The (dual) principal graph $\Gamma_\pm$ of $P_\pm$ is said to be
\emph{stable at depth $n$} if
every vertex at depth $n$ connects to at most one vertex at depth $n+1$,
no two vertices at depth $n$ connect to the same vertex at depth $n+1$,
and all edges between depths $n$ and $n+1$ are simple.
We say $(\Gamma_+,\Gamma_-)$  is \emph{stable at depth $n$} if both $\Gamma_+$ and $\Gamma_-$ are stable at depth $n$.
\end{defn}

\begin{defn}[Popa's Stability Criterion]
We say $P_+$ is \emph{stable at depth $n$} if and only if
$$
P_{n+1,+} = P_{n,+} + P_{n,+} e_{n,+} P_{n,+},
$$
where we identify $P_{n,\pm}$ with its image in $P_{n+1,\pm}$ under the right inclusion (see Definition \ref{defn:BratteliDiagram}).
We say $P_\bullet$ is \emph{stable at depth $n$}
if both $P_+$ and $P_-$ are stable at depth $n$.
\end{defn}

\begin{remark}
We remark that $P_{n,+} + P_{n,+} e_{n,+} P_{n,+}$
is the set of linear combination of diagrams of the form
$$
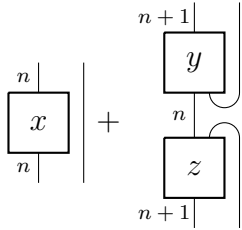
\begin{tikzpicture}[baseline = -.1cm]
	\draw (0,.8)--(0,-.8);
	\node at (-.2,.6) {{\scriptsize{$n$}}};
	\node at (-.2,-.6) {{\scriptsize{$n$}}};
	\filldraw[unshaded,thick] (-.4,.4)--(.4,.4)--(.4,-.4)--(-.4,-.4)--(-.4,.4);
	\node at (0,0) {$x$};
	\draw (.6,.8)--(.6,-.8);
\end{tikzpicture}
\,+
\begin{tikzpicture}[baseline = -.7cm]
	\draw (0,1)--(0,-2);
	\node at (-.4,.8) {{\scriptsize{$n+1$}}};
	\node at (-.4,-1.8) {{\scriptsize{$n+1$}}};
	\node at (-.2,-.5) {{\scriptsize{$n$}}};
	\filldraw[unshaded,thick] (-.4,.6)--(.4,.6)--(.4,-.2)--(-.4,-.2)--(-.4,.6);
	\node at (0,.2) {$y$};
	\filldraw[unshaded,thick] (-.4,-.8)--(.4,-.8)--(.4,-1.6)--(-.4,-1.6)--(-.4,-.8);
	\node at (0,-1.2) {$z$};
	\draw (.2,-.2) arc (180:360:.2cm) -- (.6,1);
	\draw (.2,-.8) arc (180:0:.2cm) -- (.6,-2);
\end{tikzpicture}
$$
where $x,y,z\in P_{n,\pm}$. We say $P_\bullet$ is \emph{stable at depth $n$} if both $P_+$ and $P_-$ are stable at depth $n$.
\end{remark}

\begin{lem}\label{lem:AngledBrackets}
When we identify $P_{n,\pm}$ with its image in $P_{n+1,\pm}$ by adding one vertical string to the right,
$$
P_{n,\pm} + P_{n,\pm} e_{n,\pm} P_{n,\pm} = \langle P_{n,\pm}, TL_{n+1,\pm} \rangle,
$$
where the angled brackets denote the algebra generated by $P_{n,\pm}$ and $TL_{n+1,\pm}$ under the usual multiplication.
\end{lem}

\begin{proof}
Let $e_{1,\pm},\dots,e_{n,\pm}$ be
the standard algebra generators of $TL_{n+1,\pm}$.
All of these lie in $P_{n,\pm}$ except for $e_{n,\pm}$,
so
$$
\langle P_{n,\pm},TL_{n+1,\pm} \rangle = \langle P_{n,\pm},e_{n,\pm} \rangle.
$$
For any $x \in P_{n,\pm}$,
we have $e_{n,\pm} x e_{n,\pm} = E_{P_{n-1,\pm}}(x) e_{n,\pm}$,
where $E_{P_{n-1,\pm}}(x)$ is the conditional expectation (partial trace) of $x$.
We can use this to reduce any word in $P_{n,\pm}$ and $e_{n,\pm}$
until it has at most one occurrence of $e_{n,\pm}$.
\end{proof}

The following is \cite[Proposition 4.3, Corollary 4.4]{MR1334479}.
We include a short proof for the reader's convenience.

\begin{prop}[Popa]\label{prop:stable}
The following are equivalent:
\begin{enumerate}
\item[(1)] $P_\pm$ is stable at depth $n$.
\item[(2)] $\Gamma_\pm$ is stable at depth $n$.
\end{enumerate}
\end{prop}

\begin{proof}
As in Definition \ref{defn:BratteliDiagram}, let $z_{n+1,\pm}$ be the central support of $e_{n,\pm}$ in $P_{n+1,\pm}$, and identify $P_{n,\pm}$ with its image in $P_{n+1,\pm}$ under the right inclusion. Then
\begin{align*}
P_{\pm}\text{ is stable at depth }n
& \Longleftrightarrow P_{n+1,\pm}=P_{n,\pm}+P_{n,\pm}e_{n,\pm}P_{n,\pm}\\
& \Longleftrightarrow (1-z_{n+1,\pm})P_{n+1,\pm}=(1-z_{n+1,\pm})P_{n,\pm}\\
& \Longleftrightarrow \Gamma_\pm \text{ is stable at depth }n.
\end{align*}
\end{proof}

\begin{defn}
Let $\Gamma_\pm(k)$ be the truncation of $\Gamma_\pm$ to depth $k$ consisting of all vertices with depth at most $k$ and all edges connecting them.

If $\Gamma_\pm$ is stable at depth $k$ for all $k \ge n$
then $\Gamma_\pm$ can be obtained
by attaching graphs of type $A$ to $\Gamma_\pm(n)$.
The following theorem implies that,
with some simple exceptions,
these attached graphs of type $A$ have finite length.
We call them \emph{$\Afinite$ tails}.
\end{defn}

\begin{thm}[\cite{MR1356624}]\label{thm:PopaStability}
If a connected component of $\Gamma_\pm\setminus \Gamma_\pm(n)=A_\infty$ for some $n\geq 0$, then $\Gamma_\pm\in \{ A_\infty,A_{\infty,\infty}, D_\infty\}$.
\end{thm}

Note that this theorem also follows from Theorem 6.5 in \cite{MR2679382}, which applies to infinite depth subfactors by \cite{tvc}.

%%%%%%%%%%%%%%%%%%%%%%%%%%%%%%%%%%%%%%%%%%%
\section{Principal graph stability via trains}\label{sec:ExtendingPopa}
 
In this section, we go through the proof of Popa's Principal Graph Stability Theorem \ref{thm:PrincipalGraphStability} via planar algebras and trains to prove Theorem \ref{thm:EvenStability}.

%%%%%%%%%%%%%%%%%%%%%%%%%%%%%%%%%%%%%%%%%%%
\subsection{Trains}\label{sec:Trains}

Let $P_\bullet$ be a subfactor planar algebra, and let $TL_\bullet\subset P_\bullet$ be its Temperley-Lieb planar subalgebra.

\begin{defn}
Given a set $\cS_\pm\subset P_{n,\pm}$, a \emph{train from $\cS_\pm$} is a planar tangle $\cT$ labeled by elements from $\cS_\pm$ such that for each input disk of $\cT$, its distinguished interval meets the region that meets the distinguished interval of the output disk. 
A train in $P_{k,\pm}$ can be drawn in the form 
$$
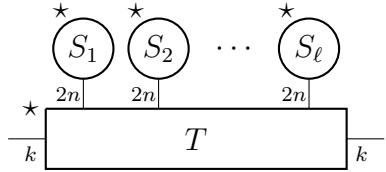
\begin{tikzpicture}[baseline = -.3cm]
	\draw (-2,-1.2)--(3,-1.2);
	\draw (-1,0)--(-1,-1);
	\draw (0,0)--(0,-1);
	\draw (2,0)--(2,-1);
	\filldraw[thick, unshaded] (-1.5,-.8)--(-1.5,-1.6)--(2.5,-1.6)--(2.5,-.8)--(-1.5,-.8);
	\draw[thick, unshaded] (-1,0) circle (.4);
	\draw[thick, unshaded] (0,0) circle (.4);	
	\draw[thick, unshaded] (2,0) circle (.4);
	\node at (-1.7,-.8) {$\star$};
	\node at (-1.3,.5) {$\star$};	
	\node at (-.3,.5) {$\star$};
	\node at (1.7,.5) {$\star$};
	\node at (1,0) {$\cdots$};
	\node at (.5,-1.2) {$T$};
	\node at (-1,0) {$S_1$};
	\node at (0,0) {$S_2$};
	\node at (2,0) {$S_\ell$};
	\node at (-1.7,-1.4) {{\scriptsize{$k$}}};
	\node at (2.7,-1.4) {{\scriptsize{$k$}}};
	\node at (-1.2,-.6) {{\scriptsize{$2n$}}};
	\node at (-.2,-.6) {{\scriptsize{$2n$}}};
	\node at (1.8,-.6) {{\scriptsize{$2n$}}};
\end{tikzpicture}
$$
where $S_1,\dots, S_\ell\in \cS_\pm$, $T$ is a single Temperley-Lieb diagram, and the distinguished interval of the external disk is at the top.

An \emph{$\ell$-car train from $\cS_\pm$}
is a train from $\cS_\pm$ with $\ell$ labeled input disks.
Note that any single diagram from $TL_{k,\pm}$ is a $0$-car train from $\cS_\pm$.
We let $\trains_{k,\pm}(\cS_\pm)$ denote the set of trains from $\cS_\pm$ in $P_{k,\pm}$.
We say \emph{trains from $\cS_\pm$ span $P_\pm$} if $P_{k,\pm}=\spann(\trains_{k,\pm}(\cS_\pm))$ for all $k\geq n$.
\end{defn}

\begin{lem}\label{lem:FourierTransform}
Suppose $k > n$.
If trains from $P_{n,+}$ span $P_{k,+}$,
then trains from $P_{n+1,-}$ span $P_{k,-}$.
\end{lem}
\begin{proof}
Consider the Fourier transform (one click rotation) of a train from $P_{n,+}$, which can be drawn with an arc passing over the $\ell$ labelled disks
$S_1,\dots,S_\ell \in P_{n,+}$.
We then combine each $S_i$ with a segment of this arc
to obtain $j(S_i) \in P_{n+1,-}$,
and thus obtain a train from $P_{n+1,-.}$.
For example,
in the case of a $3$-car train,
we have the following:
$$
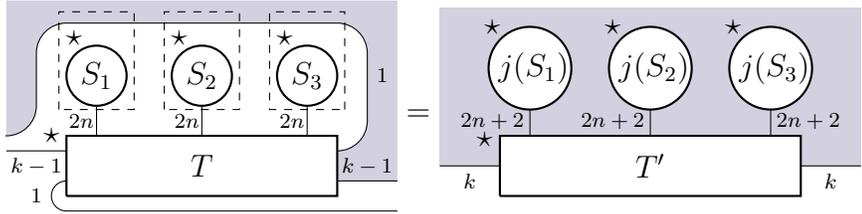
\begin{tikzpicture}[baseline = -.6cm]
	\clip (-2.6,-2)--(2.6,-2)--(2.6,1)--(-2.6,1);
	% shade almost everywhere
	\fill[shaded] (-2.6,-1)--(-2.6,1)--(2.6,1)--(2.6,-1.4)--(1,-1.4)--(1,-1.2)--(-1.6,-1.2)--(-1.6,-1);
	% strand looping over
	\filldraw[unshaded] (1.4,-1.4)--(1.8,-1) arc (-90:0:.4cm)--(2.2,.3) arc (0:90:.4cm)--(-1.8,.7) arc (90:180:.4cm) -- (-2.2,-.4) arc (0:-90:.4cm)--(-2.8,-.8)--(-2.8,-1.2);
        % strand looping under
	\draw (-1.8,-1.4) arc (90:270:.2cm) -- (2.6,-1.8);
        % horizontal lines from the TL diagram
	\draw (-2.6,-1)--(-1.6,-1);
	\draw (1,-1.4)--(2.6,-1.4);	
	% vertical lines up from the TL diagram
	\draw (-1.4,0)--(-1.4,-1);
	\draw (0,0)--(0,-1);
	\draw (1.4,0)--(1.4,-1);
	\filldraw[thick, unshaded] (-1.8,-1.6)--(1.8,-1.6)--(1.8,-.8)--(-1.8,-.8)--(-1.8,-1.6);
	\draw[thick, unshaded] (-1.4,0) circle (.4);
	\draw[thick, unshaded] (0,0) circle (.4);	
	\draw[thick, unshaded] (1.4,0) circle (.4);
	\node at (-2,-.8) {$\star$};
	\node at (-1.7,.5) {$\star$};	
	\node at (-.3,.5) {$\star$};
	\node at (1.1,.5) {$\star$};
	\node at (0,-1.2) {$T$};
	\node at (-1.4,0) {$S_1$};
	\node at (0,0) {$S_2$};
	\node at (1.4,0) {$S_3$};
	\node at (2.4,0) {{\scriptsize{$1$}}};
	\node at (-2.2,-1.6) {{\scriptsize{$1$}}};
	\node at (-2.2,-1.2) {{\scriptsize{$k-1$}}};
	\node at (2.2,-1.2) {{\scriptsize{$k-1$}}};
	\node at (-1.6,-.6) {{\scriptsize{$2n$}}};
	\node at (-.2,-.6) {{\scriptsize{$2n$}}};
	\node at (1.2,-.6) {{\scriptsize{$2n$}}};
	\draw[dashed] (-.5,-.45)--(-.5,.85)--(.5,.85)--(.5,-.45)--(-.5,-.45);
	\draw[dashed] (-1.9,-.45)--(-1.9,.85)--(-.9,.85)--(-.9,-.45)--(-1.9,-.45);
	\draw[dashed] (1.9,-.45)--(1.9,.85)--(0.9,.85)--(0.9,-.45)--(1.9,-.45);
\end{tikzpicture}
=
\begin{tikzpicture}[baseline = -.6cm]
	% shade the top half
	\fill[shaded] (-2.8,.9)--(-2.8,-1.2)--(2.8,-1.2)--(2.8,.9);
	% horizontal and three vertical strands from TL diagram
	\draw (-2.8,-1.2)--(2.8,-1.2);
	\draw (-1.6,0)--(-1.6,-1);
	\draw (0,0)--(0,-1);
	\draw (1.6,0)--(1.6,-1);
	\filldraw[thick, unshaded] (-2,-1.6)--(2,-1.6)--(2,-.8)--(-2,-.8)--(-2,-1.6);
	\draw[thick, unshaded] (-1.6,.1) circle (.55);
	\draw[thick, unshaded] (0,.1) circle (.55);
	\draw[thick, unshaded] (1.6,.1) circle (.55);
	\node at (-2.2,-.85) {$\star$};
	\node at (-2.1,.65) {$\star$};	
	\node at (-.5,.65) {$\star$};
	\node at (1.1,.65) {$\star$};
	\node at (0,-1.2) {$T'$};
	\node at (-1.55,.1) {$j(S_1)$};
	\node at (.05,.1) {$j(S_2)$};
	\node at (1.65,.1) {$j(S_3)$};
	\node at (-2.4,-1.4) {{\scriptsize{$k$}}};
	\node at (2.4,-1.4) {{\scriptsize{$k$}}};
	\node at (-2.1,-.6) {{\scriptsize{$2n+2$}}};
	\node at (-.5,-.6) {{\scriptsize{$2n+2$}}};
	\node at (2.1,-.6) {{\scriptsize{$2n+2$}}};
\end{tikzpicture}.
$$
Since trains from $P_{n,+}$ span $P_{k,+}$,
and the one click rotation is a vector space isomorphism,
it follows that trains from $P_{n+1,-}$ span $P_{k,-}$.
\end{proof}

%%%%%%%%%%%%%%%%%%%%%%%%%%%%%%%%%%%%%%%%%%%
\subsection{Trains and stability}\label{sec:TrainsAndStability}

Trains first appeared in \cite{MR1334479}.
The following lemma
allows us to translate between
the above planar algebra definition of trains
and Popa's $\lambda$-lattice formalism.

\begin{lem} \label{lem:AlgebraTrain}
For all $k>n$, 
$\spann(\trains_{k,+}(P_{n,+}))=\langle P_{n,+},TL_{k,+}\rangle$.
\end{lem}

Here, $P_{n,+}$ is considered as a subalgebra of $P_{k,+}$
by the inclusion operation of adding $k-n$ vertical strands on the right,
and the angled brackets denote the associative algebra
generated by $P_{n,+}$ and $TL_{k,+}$ under the usual multiplication.

\begin{proof}
The inclusion
$$\langle P_{n,+},TL_{k,+}\rangle \subseteq \spann(\trains_{k,+}(P_{n,+}))$$
is obvious.
For the other inclusion,
suppose
$$X \in \trains_{k,+}(P_{n,+}).$$
is an $\ell$-car train.
Thus $X$ consists of $\ell$ labelled disks $S_1,\dots,S_\ell \in P_{n,+}$
attached to a Temperley-Lieb diagram $T \in TL_{k+n\ell,+}$.

To help with the proof,
we define a metric on the regions of $T$.
If $x$ and $y$ are two points in the regions of $T$, i.e., they are not in the strings of $T$,
a path in $T$ from $x$ to $y$ is a {\em geodesic}
if it crosses the strings of $T$ transversely and a minimum number of times.
The {\em distance} $d(x,y)$
is the number of crossings in a geodesic from $x$ to $y$.
This determines a metric on the regions of $T$.
Note that it is the same as
the graph metric on the tree dual to $T$.
We will make use of the properties of metrics on trees.

Consider the case $\ell = 1$.
Thus $X$ consists of some $S \in P_{n,+}$
attached to a Temperley-Lieb diagram $T \in TL_{k+n,+}$.
Draw $X$ so that $T$ is in a rectangle
with $k$ endpoints at the top,
$k$ endpoints at the bottom,
and $S$ attached to the left edge.

Let $x$, $z$ and $p$ be points on the boundary of $T$,
where $x$ is the bottom left corner,
$z$ is the top left corner,
and $p$ is on the right edge.
Note that
\begin{itemize}
\item $d(x,z) \le 2n$,
\item $d(x,p), d(z,p) \le k$,
\end{itemize}
and furthermore, each of these inequalities is an equality modulo $2$.
By basic properties of metrics on trees,
there exists a point $y$ such that
\begin{itemize}
\item $d(y,x), d(y,z) \le n$,
\item $d(y,p) \le k-n$,
\end{itemize}
and furthermore,
such that each of these inequalities is an equality modulo $2$.
Indeed,
such a $y$ exists somewhere on the $Y$-shaped tree
containing $x$, $z$ and $p$,
and can be found by some straightforward checking of cases.

Now define an embedded graph $Y$ in $T$
consisting of three lines
going from $y$ to $x$, $z$ and $p$.
Furthermore, assume these lines cross the strings of $T$
$n$, $n$, and $k-n$ times, respectively.
To do this,
start with geodesics,
and introduce switch-backs if necessary to add more crossings.

Our graph $Y$ decomposes the diagram $X$ into three sub-diagrams.
Call these $A$, $B$ and $C$,
where $A$ is at the bottom,
$B$ is on the left and includes $S$,
and $C$ is at the top.
Then $A,C \in TL_{k,+}$, $B \in P_{n,+}$, and
$X = ABC$ is of the desired form.

Now for the induction step, suppose $\ell > 1$.
Thus $X$ consists of some $S_1,\dots,S_\ell \in P_{n,+}$
attached to a Temperley-Lieb diagram $T \in TL_{k+n\ell,+}$.
Draw the train ``sideways'' so that $T$ has
$k$ endpoints at the top,
$k$ endpoints at the bottom,
and $S_1,\dots,S_\ell$ attached along the left edge.
Number $S_1,\dots,S_\ell$ from bottom to top.

Let $x_0,\dots,x_\ell$ be points on the left edge of $T$
separating $S_1,\dots,S_\ell$.
Here, $x_0$ is the bottom left corner,
$x_\ell$ is the top left corner,
and every other $x_i$ separates $S_i$ and $S_{i+1}$.
Let $p$ lie on the right edge of $T$.

Note that
\begin{itemize}
\item $d(x_i,x_{i+1}) \le 2n$ for all $i \in \{0,\dots,\ell-1\}$, and
\item $d(x_0,p), d(x_\ell,p) \le k$,
\end{itemize}
and furthermore, each of these inequalities is an equality modulo $2$.

\item[\text{\underline{ Case 1:}}] Suppose $d(x_0,x_\ell) \le 2n$.
Draw an embedded arc from the bottom left to the top left corner of $T$
that crosses the strings of $T$ transversely $2n$ times.
This arc divides $X$ into two regions,
call them $S'$ to the left, and $T'$ to the right.
Then $S' \in P_{n,+}$ and $T' \in TL_{k+n,+}$.
Thus $X$ is a $1$-car train,
and the result follows from the base case.

\item[\text{\underline{ Case 2:}}] Suppose $d(x_i,p) \le k$ for some $i \in \{1,\dots,\ell-1\}$.
Draw an embedded arc from $x_i$ to the right edge of $T$
that crosses the strings of $T$ transversely $k$ times.
This divides $X$ into two regions,
call them $A$ on the bottom and $B$ on the top.
Then $A$ is an $i$-car train,
$B$ is a $(\ell-i)$-car train,
and $X = AB$.
The result now follows by induction.

\item[]
It remains only to show that one of the above two cases must hold.
This follows from basic properties of metrics on trees.
We prove it as a separate technical lemma below.
\end{proof}

\begin{lem}
Suppose $x_1,\dots,x_\ell$ and $p$ are vertices in a tree such that
\begin{itemize}
\item $d(x_i,x_{i+1}) \le 2n$ for all $i \in \{0,\dots,\ell-1\}$, and
\item $d(x_0,p), d(x_\ell,p) \le k$.
\end{itemize}
Then either
\begin{itemize}
\item $d(x_0,x_\ell) \le 2n$, or
\item $d(x_i,p) \le k$ for some $i \in \{1,\dots,\ell-1\}$.
\end{itemize}
\end{lem}
\begin{proof}
If $d(x_{i-1},x_{i+1}) \le 2n$ for some $i$, 
then we can omit $x_i$ from the sequence
and the hypotheses will still hold.
Thus, without loss of generality,
$$d(x_{i-1},x_{i+1}) > 2n \text{ for all }i \in \{1,\dots,\ell-1\}.$$
Under this assumption,
we show
$d(x_i,p) \le k$ for {\em every} $i \in \{1,\dots,\ell-1\}$.

Fix $i \in \{1,\dots,\ell-1\}$.
The geodesics connecting the three points $x_{i-1},x_i,x_{i+1}$
form a $Y$-shaped subtree,
and the spoke ending at $x_i$ must be
the shortest of the three spokes.
It follows that
$$d(x_i,p) < \max(d(x_{i-1},p),d(x_{i+1},p)).$$
Thus largest value of $d(x_j,p)$ occurs when either $j=0$ or $j=\ell$.
In particular,
$d(x_i,p) \leq k$ for all $i \in \{1,\dots,\ell-1\}$.
\end{proof}

We now summarize our results on trains and stability in the following theorem, which follows by a simple induction argument together with Lemma \ref{lem:AngledBrackets}, Proposition \ref{prop:stable}, and Lemma \ref{lem:AlgebraTrain}.

\begin{thm}\label{thm:equivalent}
The following are equivalent:
\item[(1)] $P_\pm$ is stable at depth $n,n+1,\dots,k-1$,
\item[(2)] $\Gamma_\pm$ is stable at depth $n,n+1,\dots,k-1$.
\item[(3)] $P_{k,\pm}=\langle P_{n,\pm},TL_{k,\pm}\rangle$, and
\item[(4)] Trains from $P_{n,\pm}$ span $P_{k,\pm}$.
\end{thm}

%%%%%%%%%%%%%%%%%%%%%%%%%%%%%%%%%%%%%%%%%%%%%%%%%%%%%%%%%%%%
\subsection{Trains and jellyfish}\label{sec:TrainsAndJellyfish}

\begin{lem}\label{lem:TwoStrandTrainsSpan}
Suppose $\cS_+\subset P_{n,+}$ generates $P_\bullet$ as a planar algebra. 
Then trains from $\cS_+$ span $P_+$ if and only if
$$
j^2(S)=\jellyfishSquared{S} \in \spann(\trains_{n+2,+}(\cS_+))
$$
for all $S\in\cS_+$.
\end{lem}

\begin{proof}
The ``only if" direction is trivial. 
The ``if" direction is the first part of the jellyfish algorithm from Section 4 of \cite{0909.4099}. 
Suppose $\cS_+$ satisfies the above \emph{jellyfish relations}.
Given an element of $P_{k,+}$ that is a tangle labeled by elements of $\cS_+$,
we use the jellyfish relation to pull a copy of $S\in\cS_+$ closer to
the region that touches the distinguished interval of the outside boundary. 
This will typically give a linear combination of labeled tangles
that contain more elements $S\in\cS_+$.
Nevertheless,
the algorithm terminates with an element of $\spann(\trains_{k,+}(\cS_+))$.
\end{proof}

\begin{lem}\label{lem:OneStrandTrainsSpan}
Suppose $\cS_+ \subset P_{n,+}$,
$\cS_- \in P_{n,-}$,
and $\cS = \cS_+ \cup \cS_-$ generates $P_\bullet$ as a planar algebra.
Then trains from $\cS$ span $P_\bullet$ if and only if
$$
j(S_\pm)= \jellyfish{S_\pm} \in \spann(\trains_{n+1,\mp}(\cS_\mp))
$$
for all $S_\pm \in \cS_\pm$.
\end{lem}

\begin{proof}
This is similar to the proof of Lemma \ref{lem:TwoStrandTrainsSpan}.
\end{proof}

\begin{prop}\label{prop:JellyfishStability}
Suppose $P_\bullet$ is generated as a planar algebra by $P_{n,+}$.
Then
\begin{enumerate}
\item[(1)] If $(\Gamma_+,\Gamma_-)$ is stable at depth $n$, then $(\Gamma_+,\Gamma_-)$ is stable at depth $k$ for all $k\geq n$.
\item[(2)] If $\Gamma_+$ is stable at depths $n$ and $n+1$, then $(\Gamma_+,\Gamma_-)$ is stable at depth $k$ for all $k\geq n+1$.
\end{enumerate}
\end{prop}

\begin{proof}
\item[(1)]
The first statement is \cite[Proposition 4.2]{MR1334479}.
In our terminology,
the proof is as follows.
Suppose $(\Gamma_+,\Gamma_-)$ is stable at depth $n$.
If $S\in P_{n,\pm}$ then,
by Theorem \ref{thm:equivalent},
$$
j(S) \in P_{n+1,\pm}  =\langle P_{n,\pm}, TL_{n+1,\pm} \rangle =\spann(\trains_{n+1,\pm}(P_{n,\pm})).
$$
By Lemma \ref{lem:OneStrandTrainsSpan},
trains from $P_{n,\pm}$ span $P_\pm$,
so again by Theorem \ref{thm:equivalent},
$(\Gamma_+,\Gamma_-)$ is stable at depth $k$ for all $k\geq n$.

\item[(2)] Suppose $\Gamma_+$ is stable at depths $n$ and $n+1$.
If $S \in P_{n,+}$ then,
by Theorem \ref{thm:equivalent},
$$
j^2(S)\in P_{n+2,+} = \langle P_{n,+}, TL_{n+2,+} \rangle = \spann(\trains_{n+2,+}(P_{n,+})).
$$
By Lemma \ref{lem:TwoStrandTrainsSpan}, trains from $P_{n,+}$ span $P_+$, so again
by Theorem \ref{thm:equivalent},
$\Gamma_+$ is stable at depth $k$ for all $k\geq n$.

It remains to show that $\Gamma_-$ is stable at depth $k$ for all $k\geq n+1$. Since trains from $P_{n,+}$ span $P_+$, trains from $P_{n+1,-}$ span $P_-$ by Lemma \ref{lem:FourierTransform}. Once more, by Theorem \ref{thm:equivalent}, $\Gamma_-$ is stable at depth $k$ for all $k\geq n+1$.
\end{proof}

%%%%%%%%%%%%%%%%%%%%%%%%%%%%%%%%%%%%%%%%%%%%%%%%%%%%%%%%%%%%
\subsection{The proof of Theorem \ref{thm:EvenStability}}\label{sec:Proof}

The proof of Popa's Principal Graph Stability Theorem has three main ingredients. First, the stability of $(\Gamma_+,\Gamma_-)$ is used in Proposition  \ref{prop:PopaSubalgebra} to construct a planar subalgebra $Q_\bullet\subset P_\bullet$ whose principal graphs $(\Lambda_+,\Lambda_-)$ are stable at all higher depths. Second, by Theorem \ref{thm:PopaStability}, the main result of \cite{MR1356624}, $\Lambda_\pm$ has no $A_\infty$ tails, so $Q_\bullet$ is finite depth. Finally, the graph norm argument in Theorem \ref{thm:GraphStability} (and Corollary \ref{cor:GraphStability}) shows $\Gamma_+=\Lambda_+$, so $Q_\bullet = P_\bullet$. Theorem \ref{thm:GraphStability} is distilled from the last statement in the proof of Popa's Principal Graph Stability Theorem. We provide a proof for the convenience of the reader.

Since we are proving an analogous result, we proceed in the same manner, but we will use the 1-click rotation argument from Lemma \ref{lem:FourierTransform} in a crucial way.

Recall that $\Gamma_\pm(k)$ is the truncation of $\Gamma_\pm$ to depth $k$. The first part of the following proposition is similar to \cite[Proposition 4.1]{MR1334479}.

\begin{prop}\label{prop:PopaSubalgebra}
Suppose $P_\bullet$ is a subfactor planar algebra, and fix $n\geq 0$. Let $Q_\bullet$ be the planar subalgebra generated by $P_{n,+}$. Let $(\Lambda_+,\Lambda_-)$ be the principal and dual principal graph of $Q_\bullet$, and note that $\Lambda_\pm(n)=\Gamma_\pm(n)$.
\begin{enumerate}
\item[(1)]
If $(\Gamma_+,\Gamma_-)$ is stable at depth $n$,
then $\Lambda_\pm(n+1)=\Gamma_\pm(n+1)$,
and $(\Lambda_+,\Lambda_-)$ is stable at depth $k$ for all $k\geq n$.
\item[(2)]
If $\Gamma_+$ is stable at depths $n$ and $n+1$,
then $\Lambda_+(n+2)=\Gamma_+(n+2)$,
$\Lambda_+$ is stable at depth $j$ for all $j\geq n$,
and $\Lambda_-$ is stable at depth $k$ for all $k\geq n+1$.
\end{enumerate}
\end{prop}
\begin{proof}
\item[(1)]
Since $(\Gamma_+,\Gamma_-)$ is stable at depth $n$, by Proposition \ref{prop:stable},
$$P_{n+1,\pm}=\spann(\trains_{n+1,\pm}(P_{n,\pm}))=Q_{n+1,\pm},$$
and thus  $\Lambda_\pm(n+1)=\Gamma_\pm(n+1)$. 
Since $Q_\bullet$ is generated by $Q_{n,+}=P_{n,+}$ and $(\Lambda_+,\Lambda_-)$ is stable at depth $n$, trains from $Q_{n,\pm}$ span $Q_{\pm}$, and $\Lambda_\pm$ is stable at depth $k$ for all $k\geq n$ by Proposition \ref{prop:JellyfishStability}.
\item[(2)]
Since $\Gamma_+$ is stable at depths $n$ and $n+1$, by Theorem \ref{thm:equivalent},
$$P_{n+2,+}=\langle P_{n,+},TL_{n+2,+}\rangle=\spann(\trains_{n+2,+}(P_{n,+}))=Q_{n+2,+},$$
and thus $\Lambda_+(n+2)=\Gamma_+(n+2)$. 
Since $Q_\bullet$ is generated by $Q_{n,+}=P_{n,+}$ and $\Lambda_+$ is stable at depths $n$ and $n+1$, trains from $Q_{n,+}$ span $Q_{+}$,  $\Lambda_+$ is stable at depth $j$ for all $j\geq n$, and $\Lambda_-$ is stable at depth $k$ for all $k\geq n+1$
by Proposition \ref{prop:JellyfishStability}.
\end{proof}

\begin{thm}\label{thm:GraphStability}
Suppose $\Lambda$ and $\Gamma$ are
finite, connected bipartite graphs with basepoints
and have the same norm $\delta > 2$.
Suppose we have Frobenius-Perron eigenvectors $\lambda$ and $\gamma$ for $\Lambda$ and $\Gamma$ respectively and there is some $n\geq 1$ such that
\begin{itemize}
\item $\Lambda(n)=\Gamma(n)\neq A_{n+1}$,
\item $\lambda|_{\Lambda(n)}=\gamma|_{\Gamma(n)}$, and
\item $\Lambda$ is stable at depth $k$ for all $k\geq n$.
\end{itemize} 
Then $\Lambda=\Gamma$.
\end{thm}

\begin{proof}
Fix a vertex $a_1$ of depth exactly $n$ in $\Lambda$.

First, suppose $a_1$ has no adjacent vertices of depth $n+1$ in $\Lambda$.
Now $\delta \lambda(a_1)$ is the sum of the values of $\lambda$
over vertices adjacent to $a_1$.
But $a_1$ and all vertices adjacent to it lie in $\Lambda(n) = \Gamma(n)$,
and $\gamma = \lambda$ when restricted to $\Gamma(n)$.
Thus $a_1$ also has no adjacent vertices of depth $n+1$ in $\Gamma$.

Now suppose $a_1$ has an adjacent vertex $a_2$ of depth $n+1$ in $\Lambda$.
Since $\Lambda$ is stable at depth $n$ and higher,
$a_1$ is attached to an $\Afinite$ tail $a_1,\dots,a_k$ in $\Lambda$.
The values of $\lambda(a_i)$ for all $i$
are determined by the values of $\delta$ and $\lambda(a_k)$.
The most important property for us is
$$\delta \lambda(a_{i+1}) < 2 \lambda(a_i)$$
for $i = 1,\dots,k-1$.

Now consider the set of vertices $b$ in $\Gamma$
that are adjacent to $a_1$ and have depth $n+1$.
The sum of the values of $\gamma$ over these vertices
is equal to $\lambda(a_2)$.
If there are at least two such vertices,
or one with multiplicity at least two,
then one of them must satisfy
$$\gamma(b) \le \lambda(a_2)/2.$$
But then
$$
\delta \gamma(b) 
\leq \delta \lambda(a_2)/2
<\lambda(a_1)
= \gamma(a_1).
$$
This contradicts the fact that $\delta \gamma(b)$ is
the sum of the values of $\gamma$ over the vertices adjacent to $b$.
It follows that $a_1$ has
exactly one adjacent vertex at depth $n+1$ in $\Gamma$, 
which we name $a_2$,
and we have $\gamma(a_2) = \lambda(a_2)$.

Applying the same argument recursively
gives a path $a_1,a_2,\dots,a_k$ in $\Gamma$ where $a_2,\dots,a_{k-1}$ have valency two,
$a_k$ has valency one,
and $\gamma(a_i) = \lambda(a_i)$ for all $i$. 

Thus every vertex of depth $n$ in $\Gamma$
is attached to an $\Afinite$ tail with the same length
as the corresponding vertex in $\Lambda$.
We conclude that $\Gamma = \Lambda$.
\end{proof}

%\begin{remark}\label{rem:3333Zmod4}
%Note that we need the Frobenius-Perron eigenvectors to agree up to depth $n$ in the above theorem. For example, the principal graphs of Izumi's $3^{\Integer/4}$ are
%$$
%\left(
%\bigraph{bwd1v1v1v1p1p1v1x0x0p0x1x0p0x0x1v1x0x0p0x1x0p0x0x1duals1v1v1x2x3v1x3x2},
%\bigraph{bwd1v1v1v1p1p1v1x0x0p0x1x0p0x1x0v1x0x0duals1v1v1x2x3v1}
%\right).
%$$
%\end{remark}

\begin{cor}\label{cor:GraphStability}
Suppose $Q_\bullet$ is a planar subalgebra of $P_\bullet$ with $\delta>2$, and let $\Lambda_+$ be the principal graph of $Q_\bullet$. Assume that
there is an $n \geq 1$ such that
\begin{itemize}
\item $\Lambda_+(n)=\Gamma_+(n)\neq A_{n+1}$, and
\item $\Lambda_+$ is finite and stable at depth $k$ for all $k\geq n$.
\end{itemize}
Then $\Lambda_+=\Gamma_+$, so $Q_\bullet=P_\bullet$.
\end{cor}
\begin{proof}
First, the depth of $P_\bullet$ is at most the depth of $Q_\bullet$. If $Q_\bullet$ is depth $q$, then $q+1$ parallel strings factor through $q$ parallel strings, since any Pimsner-Popa basis for $Q_{q+1,+}$ over $Q_{q,+}$ is a Pimsner-Popa basis for $P_{q+1,+}$ over $P_{q,+}$. Hence $\Gamma_+$ is finite, and $\delta=\|\Lambda_+\|=\|\Gamma_+\|$ by \cite{JonesICM}.

Since $\Lambda_+(n)=\Gamma_+(n)$, $\dim(Q_{n,+})=\dim(P_{n,+})$, as both are equal to the number of loops of length $2n$ on $\Gamma_+$ starting at $\star$. Thus $Q_{n,+}=P_{n,+}$. Since the traces agree on $Q_{n,+}$ and $P_{n,+}$, the resulting Frobenius-Perron eigenvectors on $\Lambda_+$ and $\Gamma_+$ agree up to depth $n$, and the hypotheses of Theorem \ref{thm:GraphStability} are satisfied. Thus $\Lambda_+=\Gamma_+$.

Finally, by counting dimensions once more, we have $Q_{k,+}=P_{k,+}$ for all $k\geq 0$, and thus $Q_\bullet=P_\bullet$. 
\end{proof}

We now have all the tools necessary to prove Theorem \ref{thm:EvenStability}. (One can easily deduce a proof of Popa's Principal Graph Stability Theorem \ref{thm:PrincipalGraphStability} from the proof of Theorem \ref{thm:EvenStability}.)

\begin{proof}[Proof of Theorem \ref{thm:EvenStability}]
By Proposition \ref{prop:PopaSubalgebra}, there is a planar subalgebra $Q_\bullet\subseteq P_\bullet$ with principal graphs $(\Lambda_+,\Lambda_-)$ such that $\Lambda_+(n+2)=\Gamma_+(n+2)$, $\Lambda_+$ is stable at depth $j$ for all $j\geq n$, and $\Lambda_-$ is stable at depth $k$ for all $k\geq n+1$. By Theorem \ref{thm:PopaStability}, $\Lambda_\pm$ is finite and obtained from the truncation $\Lambda_\pm(n+1)$ by adding $\Afinite$ tails. Finally, by Corollary \ref{cor:GraphStability}, $\Gamma_\pm=\Lambda_\pm$, so $Q_\bullet=P_\bullet$.
 \end{proof}

%%%%%%%%%%%%%%%%%%%%%%%%%%%%%%%%%%%%%%%%%%%%%%%%%%%%%%%%%%%%%%%%%%%%%%
%%%%%%%%%%%%%%%%%%%%%%%%%%%%%%%%%%%%%%%%%%%%%%%%%%%%%%%%%%%%%%%%%%%%%%
%%%%%%%%%%%%%%%%%%%%%%%%%%%%%%%%%%%%%%%%%%%%%%%%%%%%%%%%%%%%%%%%%%%%%%
\section{Applications}\label{sec:Applications}

%%%%%%%%%%%%%%%%%%%%%%%%%%%%%%%%%%%%%%%%%%%
\subsection{Jellyfish and spokes}\label{sec:Jellyfish}

Recall that a subfactor planar algebra $P_\bullet$ is called \emph{$k$ supertransitive} if $k$ is maximal such that $TL_{k,\pm}=P_{k,\pm}$. Let $P_\bullet$ be an $(n-1)$ supertransitive subfactor planar algebra with $n<\infty$. In particular, $P_\bullet\neq TL_\bullet$. 

\begin{defn}\label{defn:TwoStrandJellyfishRelations}
We call a set $\cS_+\subset P_{n,+}$ a set of \emph{$2$-strand jellyfish generators for $P_\bullet$} if 
\begin{enumerate}
\item[(1)] (Trains span) Trains from $\cS_+$ span $P_\bullet$, and
\item[(2)] (Structure algebra) $\spann(\cS_+ \cup\{f^{(n)}\})\subseteq P_{n,+}$ is an algebra under the usual multiplication.
\end{enumerate}
\end{defn}

\begin{remark}\label{rem:ForFree}
Note that if $\cS_+$ is a set of 2-strand jellyfish generators for $P_\bullet$, then we also have
\begin{enumerate}
\item[$\bullet$] (TL-capping)
$
\begin{tikzpicture}[baseline = -.1cm]
	\draw (0,-.8)--(0,0);
	\draw (.2,0)--(.2,.8);
	\node at (.6,.6) {{\scriptsize{$n-2$}}};
	\node at (.4,-.6) {{\scriptsize{$n$}}};
	\filldraw[shaded] (0,.4) arc (0:240:.15cm);
	\draw[thick, unshaded] (0,0) circle (.4);
	\node at (0,0) {$S$};
	\node at (-.55,0) {$\star$};
\end{tikzpicture}
,
\begin{tikzpicture}[baseline = -.1cm]
	\draw (0,-.8)--(0,0);
	\filldraw[shaded] (.4,.8)--(.3,0)--(-.3,0)--(-.4,.8);
	\node at (.9,.6) {{\scriptsize{$n-2$}}};
	\node at (.4,-.6) {{\scriptsize{$n$}}};
	\filldraw[unshaded] (.2,.35) arc (0:180:.2cm);
	\draw[thick, unshaded] (0,0) circle (.4);
	\node at (0,0) {$S$};
	\node at (-.55,0) {$\star$};
\end{tikzpicture}
\in TL_{n-1,\pm}
$
for all $S\in \cS_+$.
\item[$\bullet$] (Rotational closure) For each $S\in \cS_{+}$,
$
\begin{tikzpicture}[baseline = -.1cm]
	\node at (.4,.6) {{\scriptsize{$n-2$}}};
	\node at (-.4,-.6) {{\scriptsize{$n-2$}}};
	\filldraw[shaded] (0,.8)--(0,0)--(-.2,0) arc (0:-180:.35cm) -- (-.9,.8);
	\filldraw[unshaded] (0,.8)--(0,0)--(-.3,0) arc (0:-180:.25cm) -- (-.8,.8);
	\filldraw[shaded] (0,-.8)--(0,0)--(.2,0) arc (180:0:.35cm) -- (.9,-.8);
	\filldraw[unshaded] (0,-.8)--(0,0)--(.3,0) arc (180:0:.25cm) -- (.8,-.8);
	\draw[thick, unshaded] (0,0) circle (.4);
	\node at (0,0) {$S$};
	\node at (-.55,0) {$\star$};
\end{tikzpicture}
\in \spann(\cS_+)\oplus TL_{n,+}.
$
\end{enumerate}
\end{remark}

\begin{defn}\label{defn:OneStrandJellyfishRelations}
We call a set $\cS=\cS_+\cup \cS_-$ with $\cS_\pm\subseteq P_{n,\pm}$ a set of \emph{$1$-strand jellyfish generators for $P_\bullet$} if
\begin{enumerate}
\item[(1)] (Trains span) Trains from $\cS_\pm$ span $P_\bullet$.
\item[(2)] (Structure algebra) $\spann(\cS_+ \cup\{f^{(n)}\})\subset P_{n,+}$ and $\spann(\cS_- \cup\{\check{f}^{(n)}\})\subset P_{n,-}$ are algebras under the usual multiplication.
\end{enumerate}
\end{defn}

\begin{remark}
As in Remark \ref{rem:ForFree}, if $\cS$ is a set of $1$-strand jellyfish generators, then we also have
\begin{enumerate}
\item[$\bullet$] (TL-capping)
$
\begin{tikzpicture}[baseline = -.1cm]
	\draw (0,-.8)--(0,0);
	\draw (.2,0)--(.2,.8);
	\node at (.6,.6) {{\scriptsize{$n-2$}}};
	\node at (.4,-.6) {{\scriptsize{$n$}}};
	\draw (0,.4) arc (0:240:.15cm);
	\draw[thick, unshaded] (0,0) circle (.4);
	\node at (0,0) {$S$};
	\node at (-.55,0) {$\star$};
\end{tikzpicture}
\in TL_{n-1,\pm}
$
for all $S\in \cS_\pm$.
\item[$\bullet$] (Rotational closure) For each $S\in \cS_{\pm}$,
$
\begin{tikzpicture}[baseline = -.1cm]
	\node at (.4,.6) {{\scriptsize{$n-1$}}};
	\node at (-.4,-.6) {{\scriptsize{$n-1$}}};
	\draw (0,.8)--(0,0)--(-.3,0) arc (0:-180:.25cm) -- (-.8,.8);
	\draw (0,-.8)--(0,0)--(.3,0) arc (180:0:.25cm) -- (.8,-.8);
	\draw[thick, unshaded] (0,0) circle (.4);
	\node at (0,0) {$S$};
	\node at (-.55,0) {$\star$};
\end{tikzpicture}
\in \spann(\cS_\mp)\oplus TL_{n,\mp}.
$
\end{enumerate}
\end{remark}

\begin{remark}\label{rem:OneStrandJellyfishImpliesTwoStrandJellyfish}
If $\cS=\cS_+\cup \cS_-$ is a set of $1$-strand jellyfish generators for $P_\bullet$, then $\cS_+$ is a set of $2$-strand jellyfish generators for $P_\bullet$, and $\cS_-$ is a set of $2$-strand jellyfish generators for the dual of $P_\bullet$ (obtained by reversing the shading).
\end{remark}

\begin{defn}\label{defn:spoke}
A \emph{simply laced spoke graph} is a tree with two distinguished vertices $\star$ and $c$ such that $\star$ has valence $1$ and every vertex except possibly $c$ has valence at most $2$.

In general, a \emph{spoke graph} is a graph obtained from a simply laced spoke graph by replacing some edges with multiple edges. Further, we require these multiple edges to be incident to $c$, but not include the edge from $c$ in the direction of $\star$.

For a (dual) principal graph $\Gamma$ to to be a spoke graph, we require that $\star$ be the basepoint of $\Gamma$.
\end{defn}

\begin{remark}
Since $P_\bullet$ is $n-1$ supertransitive, If $\Gamma_\pm$ is a spoke graph, then $c$ is at depth $n-1$.
\end{remark}

\begin{ex}
Some examples of finite simply laced spoke graphs are the 2221, 3311, 3333, and 4442 principal graphs:
$$
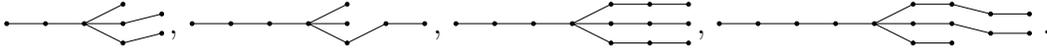
\bigraph{gbg1v1v1p1p1v1x0x0p0x1x0},\bigraph{gbg1v1v1v1p1p1v1x0x0v1},\bigraph{gbg1v1v1v1p1p1v1x0x0p0x1x0p0x0x1v1x0x0p0x1x0p0x0x1},\bigraph{gbg1v1v1v1v1p1p1v1x0x0p0x1x0p0x0x1v0x1x0p0x0x1v1x0p0x1}\,.
$$
An example of an infinite simply laced spoke graph is the $D_\infty$ principal graph
$$
\begin{tikzpicture}[baseline=-.3cm,scale=.6]
	\filldraw (0,0) circle (.05cm);
	\filldraw (1,0) circle (.05cm);
	\filldraw (2,-.3) circle (.05cm);
	\filldraw (2,.3) circle (.05cm);
	\filldraw (3,-.3) circle (.05cm);
	\filldraw (4,-.3) circle (.05cm);
	\filldraw (5,-.3) circle (.05cm);
	\node at (5.75,-.3) {$\cdots$};
	\draw[] (0,0)--(1,0)--(2,.3);
	\draw[] (1,0)--(2,-.3)--(5,-.3);
\end{tikzpicture}
$$
Examples of spoke graphs that are not simply laced are the principal graphs of fixed-point subfactors $R^{G}\subset R$ for $G$ non-abelian, e.g., $G=S_3$:
$$
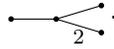
\begin{tikzpicture}[baseline=0cm,scale=.6]
	\filldraw (0,0) circle (.05cm);
	\filldraw (1,0) circle (.05cm);
	\filldraw (2,-.3) circle (.05cm);
	\filldraw (2,.3) circle (.05cm);
	\node at (1.5,-.4) {{\scriptsize{$2$}}};
	\draw[] (0,0)--(1,0)--(2,.3);
	\draw[] (1,0)--(2,-.3);
\end{tikzpicture}\,.
$$
\end{ex}

\begin{thm}\label{thm:OneStrand}
Suppose $P_\bullet$ is an $(n-1)$ supertransitive subfactor planar algebra with $\delta>2$ and principal graphs $(\Gamma_+,\Gamma_-)$.
The following are equivalent.
\begin{enumerate}
\item[(1)]
$P_{n,+}\cup P_{n,-}$ is a set of $1$-strand jellyfish generators for $P_\bullet$.
\item[(2)]
$\Gamma_+$ and $\Gamma_-$ are finite spoke graphs.
\item[(3)]
$\Gamma_+(n+1)$ and $\Gamma_-(n+1)$ are spoke graphs. 
\end{enumerate}
\end{thm}
\begin{proof}
\item[\emph{$(1)\Rightarrow (2)$:}] Since trains from $P_{n,\pm}$ span $P_\bullet$, $\Gamma_\pm$ is stable at depth $k$ for all $k\geq n$ by Theorem \ref{thm:equivalent}. By Theorem \ref{thm:PopaStability}, $\Gamma_+,\Gamma_-$ are finite.
\item[\emph{$(2)\Rightarrow (3)$:}] Trivial.
\item[\emph{$(3)\Rightarrow (1)$:}] Note that $(\Gamma_+,\Gamma_-)$ is stable at depth $n$, so let $Q_\bullet$ be the subfactor planar algebra generated by $P_{n,\pm}$ as in Proposition \ref{prop:PopaSubalgebra}, and note that trains from $P_{n,\pm}$ span $Q_\bullet$. Since $P_\bullet$ is $(n-1)$ supertransitive, $P_{n,+}\cup P_{n,-}$ is a set of 1-strand jellyfish generators for $Q_\bullet$. Finally, by Popa's Principal Graph Stability Theorem \ref{thm:PrincipalGraphStability}, $Q_\bullet=P_\bullet$.
\end{proof}

\begin{thm}\label{thm:TwoStrand}
Suppose $P_\bullet$ is an $(n-1)$ supertransitive subfactor planar algebra with $\delta>2$ and principal graph $\Gamma_+$.
The following are equivalent.
\begin{enumerate}
\item[(1)]
$P_{n,+}$ is a set of $2$-strand jellyfish generators for $P_\bullet$.
\item[(2)]
$\Gamma_+$ is a finite spoke graph, and $\Gamma_-$ is stable at depth $k$ for all $k\geq n+1$.
\item[(3)]
$\Gamma_+(n+2)$ is a spoke graph.
\end{enumerate}
\end{thm}
\begin{proof}
\item[\emph{$(1)\Rightarrow (2)$:}] Since trains from $P_{n,+}$ span $P_\bullet$, $\Gamma_+$ is stable at depth $k$ for all $k\geq n$ by Theorem \ref{thm:equivalent}. By Lemma \ref{lem:FourierTransform}, trains from $P_{n+1,-}$ span $P_-$, so again $\Gamma_-$ is stable at depth $k$ for all $k\geq n+1$. By Theorem \ref{thm:PopaStability}, $\Gamma_+$ is finite (and thus so is $\Gamma_-$).
\item[\emph{$(2)\Rightarrow (3)$:}] Trivial.
\item[\emph{$(3)\Rightarrow (1)$:}] Note that $\Gamma_+$ is stable at depths $n$ and $n+1$, so let $Q_\bullet$ be the subfactor planar algebra generated by $P_{n,\pm}$ as in Proposition \ref{prop:PopaSubalgebra}, and note that trains from $P_{n,+}$ span $Q_\bullet$. Since $P_\bullet$ is $(n-1)$ supertransitive, $P_{n,+}$ is a set of 2-strand jellyfish generators for $Q_\bullet$. Finally, by Theorem \ref{thm:EvenStability}, $Q_\bullet=P_\bullet$.
\end{proof}

\begin{remark}
If $P_\bullet$ is a subfactor planar algebra with principal graphs $(\Gamma_+,\Gamma_-)$, and if $\Gamma_+$ and $\Gamma_-$ are simply laced spoke graphs, then $\Gamma_+=\Gamma_-$. The traces of projections that are dual to each other must be equal, and thus the Frobenius-Perron dimensions of vertices of $\Gamma_+,\Gamma_-$ at odd depths must agree.
\end{remark}

\begin{cor}
There is no set of jellyfish generators in $P_{6,\pm}$ ($1$ or $2$-strand) for the Asaeda-Haagerup subfactor planar algebra \cite{MR1686551} with principal graphs
$$
(\Gamma_+,\Gamma_-)=\left(\bigraph{bwd1v1v1v1v1v1p1v1x0p0x1v1x0p0x1p0x1v1x0x0v1duals1v1v1v1x2v2x1x3v1}, \bigraph{bwd1v1v1v1v1v1p1v0x1p0x1v0x1v1duals1v1v1v1x2v1}\right).
$$
\end{cor}
\begin{proof}
This is immediate from Theorems \ref{thm:OneStrand} and \ref{thm:TwoStrand} and Remark \ref{rem:OneStrandJellyfishImpliesTwoStrandJellyfish}.
\end{proof}

\begin{prop}
Recall there are $n$ non-isomorphic subfactor planar algebras with principal graphs $(D_{n+2}^{(1)},D_{n+2}^{(1)})$ for $4\leq n<\infty$ \cite{MR1213139,MR1278111}.
$$
D_{n+2}^{(1)}=
\begin{tikzpicture}[baseline=-.3cm,scale=.6]
	\filldraw (0,0) circle (.05cm);
	\node at (0,-.4) {{\scriptsize{$0$}}};
	\filldraw (1,0) circle (.05cm);
	\node at (1,-.4) {{\scriptsize{$1$}}};
	\filldraw (2,-.3) circle (.05cm);
	\node at (2,-.7) {{\scriptsize{$2$}}};
	\filldraw (2,.3) circle (.05cm);
	\filldraw (3,-.3) circle (.05cm);
	\node at (3.75,-.3) {$\cdots$};
	\filldraw (4.5,-.3) circle (.05cm);
	\filldraw (5.5,-.3) circle (.05cm);
	\node at (5.5,-.7) {{\scriptsize{$n-1$}}};
	\filldraw (6.5,0) circle (.05cm);
	\filldraw (6.5,-.6) circle (.05cm);
	\node at (6.5,-1) {{\scriptsize{$n$}}};
	\draw[] (0,0)--(1,0)--(2,.3);
	\draw[] (1,0)--(2,-.3)--(3,-.3);
	\draw[] (4.5,-.3)--(5.5,-.3)--(6.5,0);
	\draw[] (5.5,-.3)--(6.5,-.6);
\end{tikzpicture}
$$
If $P_\bullet$ is such a subfactor planar algebra, then $P_\bullet$ is not generated by $P_{2,\pm}$.
\end{prop}
\begin{proof}
Let $Q_\bullet$ be the subfactor planar algebra generated by $P_{2,\pm}$ as in Proposition \ref{prop:PopaSubalgebra}, and note that trains from $P_{2,\pm}$ span $Q_\bullet$. If $Q_\bullet$ has principal graphs $(\Lambda_+,\Lambda_-)$, then $\Lambda_\pm$ is stable at depth $k$ for all $k\geq 2$, so $\Lambda_\pm = D_\infty$.
\end{proof}

\begin{remark}\label{rem:affineD}
In \cite{MPAffineAandD}, Morrison and Penneys give a planar algebra presentation by generators and relations for the $A_{2n-1}^{(1)}$ and $D_{n+2}^{(1)}$ planar algebras using jellyfish of different sizes. The $A_{2n-1}^{(1)}$ planar algebras are generated by one 2-box and two $n$-boxes, and the $D_{n+2}^{(1)}$ planar algebras are generated by one 2-box and one $n$-box. The differences in the relations for each of the $n$ distinct subfactor planar algebras are the rotational eigenvalues of the $n$-boxes.
\end{remark}

\begin{defn}
Recall from \cite{1007.1730} that \emph{translating} a principal graph means attaching an $A_k$ graph to the left, and \emph{extending} means adding additional edges and vertices to the right, where by convention, the basepoint $\star$ corresponding to the empty diagram is always at the left, and vertices are placed left to right corresponding to depth.
\end{defn}

\begin{cor}\label{cor:TranslatedExtension}
If $\Gamma_+$ is a translated extension of
$$
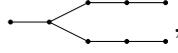
\bigraph{gbg1v1p1v1x0p0x1v1x0p0x1},
$$
then $(\Gamma_+,\Gamma_-)$ is one of
\begin{align*}
\cH&=\left(\bigraph{bwd1v1v1v1p1v1x0p0x1v1x0p0x1duals1v1v1x2v2x1}, \bigraph{bwd1v1v1v1p1v1x0p1x0duals1v1v1x2}\right)\text{ or}\\
\cE\cH&=\left(\bigraph{bwd1v1v1v1v1v1v1v1p1v1x0p0x1v1x0p0x1duals1v1v1v1v1x2v2x1}, \bigraph{bwd1v1v1v1v1v1v1v1p1v1x0p1x0duals1v1v1v1v1x2}\right).
\end{align*}
\end{cor}
\begin{proof}
By the classification of subfactors with index $4$, 
$$
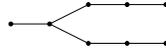
\bigraph{gbg1v1p1v1x0p0x1v1x0p0x1}
$$
is not the principal graph of a subfactor, so $\Gamma_+$ must be a nontrivial translated extension, and thus $\delta>2$. Thus by Theorem \ref{thm:TwoStrand}, $\Gamma_+$ is a finite spoke graph. By \cite{JonesICM}, the modulus of a finite depth subfactor planar algebra is equal to the norm of its principal graph, so 
$$
\delta
=\|\Gamma_+\|
<\left\|
\begin{tikzpicture}[baseline=-.25cm,scale=.7]
	\filldraw (0,0) circle (.05cm);
	\filldraw (1,-.5) circle (.05cm);
	\filldraw (1,0) circle (.05cm);
	\filldraw (1,.5) circle (.05cm);
	\filldraw (2,-.5) circle (.05cm);
	\filldraw (2,0) circle (.05cm);
	\filldraw (2,.5) circle (.05cm);
	\filldraw (3,-.5) circle (.05cm);
	\filldraw (3,0) circle (.05cm);
	\filldraw (3,.5) circle (.05cm);
	\node at (3.75,-.5) {$\cdots$};
	\node at (3.75,0) {$\cdots$};
	\node at (3.75,.5) {$\cdots$};
	\node at (-.4,0) {{\scriptsize{$1$}}};
	\node at (1,-.9) {{\scriptsize{$t^{-1}$}}};
	\node at (2,-.9) {{\scriptsize{$t^{-2}$}}};
	\node at (3,-.9) {{\scriptsize{$t^{-3}$}}};
	\draw[] (0,0)--(1,.5)--(3,.5);
	\draw[] (0,0)--(3,0);
	\draw[] (0,0)--(1,-.5)--(3,-.5);
\end{tikzpicture}
\right\|
=\sqrt{4.5}
<\sqrt{5}
$$
(by Lemma A.4 of \cite{1007.2240}, the infinite graph above has a strictly positive $\ell^2$-eigenvector whose weights are given by the labels above corresponding to eigenvalue $t+t^{-1}$ where $t=\sqrt{2}$. The norm of the infinite graph is then $t+t^{-1}$ by Theorems 4.4 and 6.2 of \cite{MR986363}).
By the classification of subfactors below index $5$ \cite{1007.1730,1007.2240,1109.3190,1010.3797}, we know $(\Gamma_+,\Gamma_-)\in \{\cH,\cE\cH\}$.
\end{proof}

\begin{remark}
The classification of subfactors to index 5 can be used to completely classify all subfactor planar algebras $P_\bullet$ of modulus $\delta>2$ whose principal graph $\Gamma_+$ is a tree with no vertices of degree greater than 3 and at most two triple points. Note that $\Gamma_+$ must be finite by Theorem \ref{thm:PopaStability}, and $2<\delta=\|\Gamma_+\|$ by \cite{JonesICM}.

If $\Gamma_+$ has exactly one triple point, then the same argument as in Corollary \ref{cor:TranslatedExtension} shows that $(\Gamma_+,\Gamma_-)\in \{\cH,\cE\cH\}$. If $\Gamma_+$ has exactly two triple points, and $(\Gamma_-,\Gamma_+)\notin \{\cH,\cE\cH\}$, then 
$$
(\Gamma_\pm,\Gamma_\mp)
=\cA\cH
=\left(\bigraph{bwd1v1v1v1v1v1p1v1x0p0x1v1x0p0x1p0x1v1x0x0v1duals1v1v1v1x2v2x1x3v1}, \bigraph{bwd1v1v1v1v1v1p1v0x1p0x1v0x1v1duals1v1v1v1x2v1}\right).
$$
To see this, note that 
$$
\left\|
\begin{tikzpicture}[baseline=-.25cm,scale=.7]
	\filldraw (0,0) circle (.05cm);
	\filldraw (1,0) circle (.05cm);
	\filldraw (-1,-.3) circle (.05cm);
	\filldraw (-1,.3) circle (.05cm);
	\filldraw (2,-.3) circle (.05cm);
	\filldraw (2,.3) circle (.05cm);
	\filldraw (-2,-.3) circle (.05cm);
	\filldraw (-2,.3) circle (.05cm);
	\filldraw (3,-.3) circle (.05cm);
	\filldraw (3,.3) circle (.05cm);
	\node at (3.75,-.3) {$\cdots$};
	\node at (3.75,.3) {$\cdots$};
	\node at (-2.75,-.3) {$\cdots$};
	\node at (-2.75,.3) {$\cdots$};
	\node at (0,-.3) {{\scriptsize{$1$}}};
	\node at (1,-.3) {{\scriptsize{$1$}}};
	\node at (2,-.7) {{\scriptsize{$t^{-1}$}}};
	\node at (3,-.7) {{\scriptsize{$t^{-2}$}}};
	\node at (-1,-.7) {{\scriptsize{$t^{-1}$}}};
	\node at (-2,-.7) {{\scriptsize{$t^{-2}$}}};
	\draw[] (0,0)--(-1,.3)--(-2,.3);
	\draw[] (0,0)--(-1,-.3)--(-2,-.3);
	\draw[] (0,0)--(1,0);
	\draw[] (1,0)--(2,.3)--(3,.3);
	\draw[] (1,0)--(2,-.3)--(3,-.3);
\end{tikzpicture}
\right\|
=\sqrt{5}
$$
(once again, the infinite graph has a strictly positive $\ell^2$-eigenvector corresponding to eigenvalue $t+t^{-1}$ where $t=\frac{1}{2}(1+\sqrt{5})$). A simple induction argument shows that if we subdivide the simple edge between the two triple points in the infinite graph above, the norm will decrease, i.e., 
$$
\left\|
\begin{tikzpicture}[baseline=-.1cm,scale=.7]
	\filldraw (-1,0) circle (.05cm);
	\filldraw (0,0) circle (.05cm);
	\filldraw (1,0) circle (.05cm);
	\filldraw (2,0) circle (.05cm);
	\filldraw (-2,-.3) circle (.05cm);
	\filldraw (-2,.3) circle (.05cm);
	\filldraw (3,-.3) circle (.05cm);
	\filldraw (3,.3) circle (.05cm);
	\filldraw (-3,-.3) circle (.05cm);
	\filldraw (-3,.3) circle (.05cm);
	\filldraw (4,-.3) circle (.05cm);
	\filldraw (4,.3) circle (.05cm);
	\node at (0.5,0) {$\cdots$};
	\node at (4.75,-.3) {$\cdots$};
	\node at (4.75,.3) {$\cdots$};
	\node at (-3.75,-.3) {$\cdots$};
	\node at (-3.75,.3) {$\cdots$};
	\draw[] (-1,0)--(-2,.3)--(-3,.3);
	\draw[] (-1,0)--(-2,-.3)--(-3,-.3);
	\draw[] (-1,0)--(0,0);
	\draw[] (1,0)--(2,0);
	\draw[] (2,0)--(3,.3)--(4,.3);
	\draw[] (2,0)--(3,-.3)--(4,-.3);
\end{tikzpicture}
\right\|
<\sqrt{5}
$$
(see \cite[3.1.2]{SpectraOfGraphs}).
Hence if $\Gamma_+$ has exactly two triple points, then $\|\Gamma_+\|<\sqrt{5}$, and the claim follows. Finally, note there is exactly one subfactor planar algebra with each of the principal graphs $\cH,\cE\cH,\cA\cH$ \cite{MR1686551,0909.4099}.
\end{remark}

%%%%%%%%%%%%%%%%%%%%%%%%%%%%%%%%%%%%%%%%%%%%%%%%%%%%%%%%%%%%%%%%%%%%%%
\subsection{Another proof of the quadratic tangles formula}\label{sec:QT}

We say $P_\bullet$ has annular multiplicities $*10$ if $\delta>2$ and $(\Gamma_+,\Gamma_-)$ is a translated extension of 
$$
\left(\bigraph{gbg1v1p1v1x0p0x1},\bigraph{gbg1v1p1v1x0p1x0}\right)
%\text{ or }
%\left(\bigraph{gbg1v1p1v1x0p1x0},\bigraph{gbg1v1p1v1x0p0x1}\right)
$$
(for further details on annular multiplicities $*10$, see \cite{MR1317352,math/1007.1158,1007.2240}).

In this case, we recover most of Theorem 5.1.11 of \cite{math/1007.1158}. 
The statement uses the following notation.
\begin{itemize}
\item $[k]=(q^k-q^{-k})/(q-q^{-1})$, where $[2]=q+q^{-1}=\delta$,
\item $n\in \Natural$ is such that $P_\bullet$ is $(n-1)$ supertransitive, 
\item $\check{r}\geq r\geq 1$ is the ratio of the projections at depth $n$ of $\Gamma_-,\Gamma_+$ respectively (by calculating Frobenius-Perron dimensions, $\check{r}=[n+2]/[n]$),
\item $S\in P_{n,+}$ and $\check{S}\in P_{n,-}$ are low-weight rotational eigenvectors with eigenvalue $\omega_S$,
\item $\set{\cup_i(S)}{0\leq i\leq 2n+1}$ is the basis of annular consequences of $S$, and $\set{\widehat{\cup}_i(S)}{0\leq i\leq 2n+1}$ is the dual basis,
\item $S\circ S = \quadratic{n}{S}{S}$ is the quadratic tangle (which lies in annular consequences),
\item $\sigma_S = \omega_S^{1/2}$, which is determined by $\check{r}\geq r\geq  1$, and
\item $W_{k,\omega_S}=q^k+q^{-k}-\omega_S-\omega_S^{-1}$.
\end{itemize}
Our proof of Jones' result only uses Jones' formulas for the dual basis $\widehat{\cup}_i(S)$'s in terms of the annular basis $\cup_i(S)$. (For annular multiplicities $*10$, $S\circ S$ lies in annular consequences, so taking inner products is easy.)

\begin{prop}\label{prop:QT}
If  $P_\bullet$ has annular multiplicities $*10$,
then there is no set of $1$-strand jellyfish generators for $P_\bullet$ in $P_{n,+}$.
Moreover, $n$ is even, and 
$$
r+\frac{1}{r}=2+\frac{2+\omega_S+\omega_S^{-1}}{[n+2][n]}.
$$
\end{prop}

\begin{proof}
The first claim is immediate from Theorem \ref{thm:OneStrand}. 

To prove the quadratic tangles constraint, note that by Theorem \ref{thm:equivalent},
$$
\begin{tikzpicture}[baseline = -.3cm]
	\filldraw[shaded] (-.7,-.8)--(-.7,0) arc (180:0:.7cm)--(.7,-.8);
	\draw (0,0)--(0,-.8);
	\node at (.2,-.6) {{\scriptsize{$2n$}}};
	\draw[thick, unshaded] (0,0) circle (.4);
	\node at (0,0) {$\check{S}$};
	\node at (-.25,-.5) {$\star$};
\end{tikzpicture}
\in 
P_{n+1,+}=\spann(\trains_{n+1,+}(\{S\})),
$$
and since $\Gamma_-$ is not a spoke graph, by Theorem \ref{thm:OneStrand},
$$
\begin{tikzpicture}[baseline = -.1cm]
	\fill[shaded] (-1,-.8)--(-1,1)--(1,1)--(1,-.8)--(-1,-.8);
	\filldraw[unshaded] (-.7,-.8)--(-.7,0) arc (180:0:.7cm)--(.7,-.8);
	\draw (0,0)--(0,-.8);
	\node at (.2,-.6) {{\scriptsize{$2n$}}};
	\draw[thick, unshaded] (0,0) circle (.4);
	\node at (0,0) {$S$};
	\node at (-.25,-.5) {$\star$};
\end{tikzpicture}
\notin
\spann(\trains_{n+1,-}(\{\check{S}\})).
$$
Hence the coefficient of $\cup_{n+1}(S)$ in $S\circ S$ must be zero, so by Propositions 4.2.9 (iii) and 4.4.1 of \cite{math/1007.1158},
\begin{align}
0&=\langle S\circ S, \widehat{\cup}_{n+1}(S)\rangle\notag\\ 
&= \sigma_S^n\frac{[2n+2]}{W_{2n+2,\omega_S}}\Tr(S^3)+\frac{[n+1]}{W_{2n+2,\omega_S}}((-\sigma_S)^{n+1}+(-\sigma_S)^{-n-1})\Tr(\check{S}^3).
\label{eqn:QT}
\end{align}
If $n$ is odd, then $\Tr(S^3)=\pm \Tr(\check{S}^3)$ (by sphericality), and thus
$$
q^{n+1}+q^{-n-1}=\frac{[2n+2]}{[n+1]}=\pm(\sigma_S+\sigma_S^{-1})\leq 2,
$$
which is impossible if $q>1$. Now substituting 
$$
\Tr(S^3)=\frac{r^{1/2}-r^{-1/2}}{[n+1]^{1/2}},\,\Tr(\check{S}^3)=\frac{\check{r}^{1/2}-\check{r}^{-1/2}}{[n+1]^{1/2}},\text{ and } \check{r}=\frac{[n+2]}{[n]}
$$
into Equation \eqref{eqn:QT}, it simplifies to 
$$
(r^{1/2}-r^{-1/2})[2n+2]-(\sigma_S+\sigma_S^{-1})\left(\left(\frac{[n+2]}{[n]}\right)^{1/2}-\left(\frac{[n+2]}{[n]}\right)^{-1/2}\right)[n+1]=0.
$$
Solving for $r^{1/2}-r^{-1/2}$ and squaring gives the desired equation after using the identity
$$
[2n+2]^2-[n+1]^2([n+2]^2+[n]^2-2[n+2][n])=0.
$$
\end{proof}

%%%%%%%%%%%%%%%%%%%%%%%%%%%%%%%%%%%%%%%%%%%%%%%%%%%%%%%%%%%%%%%%%%%%%%
%%%%%%%%%%%%%%%%%%%%%%%%%%%%%%%%%%%%%%%%%%%%%%%%%%%%%%%%%%%%%%%%%%%%%%
%%%%%%%%%%%%%%%%%%%%%%%%%%%%%%%%%%%%%%%%%%%%%%%%%%%%%%%%%%%%%%%%%%%%%%
\bibliographystyle{amsalpha}
\bibliography{../../bibliography/bibliography}
\end{document}